\documentclass[a4paper, 10pt]{amsart}

\usepackage{amscd,amsfonts, amsmath,amssymb,amsthm, bm,comment,eucal,url,mathrsfs,mathtools,MnSymbol,stmaryrd}
\usepackage{hyperref}
\usepackage{graphicx}
\usepackage[all]{xy}
\usepackage{tikz}
\usetikzlibrary{cd}
\usetikzlibrary{patterns}

\DeclareMathOperator{\bl}{bl}

\DeclareMathOperator{\diag}{diag}

\DeclareMathOperator{\fun}{fun}

\DeclareMathOperator{\id}{id}

\DeclareMathOperator{\Kill}{Kill}

\DeclareMathOperator{\Res}{Res}
\DeclareMathOperator{\rtype}{rtype}

\DeclareMathOperator{\std}{std}
\DeclareMathOperator{\Spec}{Spec}

\DeclareMathOperator{\Tor}{Tor}

\DeclareMathOperator{\Stand}{Stand}
\DeclareMathOperator{\StandType}{StandType}
\DeclareMathOperator{\Transp}{Transp}
\DeclareMathOperator{\STransp}{STransp}
\DeclareMathOperator{\GL}{GL}

\DeclareMathOperator{\SL}{SL}
\DeclareMathOperator{\SO}{SO}

\DeclareMathOperator{\Uu}{U}

\newcommand{\bC}{\mathbb{C}}

\newcommand{\bQ}{\mathbb{Q}}
\newcommand{\bR}{\mathbb{R}}
\newcommand{\bZ}{\mathbb{Z}}
\newcommand{\cA}{\mathcal{A}}
\newcommand{\cB}{\mathcal{B}}

\newcommand{\cI}{\mathcal{I}}
\newcommand{\cJ}{\mathcal{J}}

\newcommand{\cL}{\mathcal{L}}
\newcommand{\cM}{\mathcal{M}}
\newcommand{\cP}{\mathcal{P}}
\newcommand{\cO}{\mathcal{O}}

\newcommand{\cT}{\mathcal{T}}

\newcommand{\fg}{\mathfrak{g}}

\newcommand{\fS}{\mathfrak{S}}

\theoremstyle{plain}
\newtheorem{thm}{Theorem}[section]
\newtheorem{defn-prop}[thm]{Definition-Proposition}
\newtheorem{cor}[thm]{Corollary}
\newtheorem{lem}[thm]{Lemma}
\newtheorem{prop}[thm]{Proposition}
\newtheorem{var}[thm]{Variant}
\theoremstyle{definition}
\newtheorem{ass}[thm]{Assumption}

\newtheorem{cons}[thm]{Construction}
\newtheorem{defn}[thm]{Definition}
\newtheorem{ex}[thm]{Example}

\newtheorem{rem}[thm]{Remark}
\begin{document}
	\title[Uniform decomposition]{Uniform decomposition of the flag scheme by a symmetric subgroup action}
	\author{Takuma Hayashi}
	\address{Osaka Central Advanced Mathematical Institute, Osaka Metropolitan University, 3-3-138 Sugimoto, Sumiyoshi-ku Osaka 558-8585, Japan}
	\email{takuma.hayashi.forwork@gmail.com}
	\date{}
	\subjclass[2020]{14L15, 14L30, 20G35.}
	\keywords{reductive group scheme, $\theta$-fixed point subgroup scheme, orbit decomposition, descent, affine immersion.}
	\begin{abstract}
		In this paper, we establish a scheme-theoretic analog of the works of Matsuki and Richardson--Springer on the symmetric subgroup orbit decomposition of the flag variety under a certain assumption that asserts local constancy of their combinatorial description of the classification of orbits at geometric points (the local constancy hypothesis). We also prove that scheme-theoretic models of orbits are affinely imbedded into the flag scheme under the same assumption (Beilinson--Bernstein's affinity theorem). Finally, we compute some classical examples to give scheme-theoretic consequences of the geometric point free and descent phenomena of orbit decompositions.
	\end{abstract}
	
	\maketitle
	
	\section{Introduction}\label{sec:intro}
	
	In representation theory of real reductive Lie groups, the orbit decomposition of the complex flag variety of a connected complex reductive algebraic group with respect to the action by a symmetric subgroup plays an important role due to the Beilinson--Bernstein correspondence (\cite[Section 3]{MR610137}). For the duality theorem on the cohomological induction and the Harish-Chandra modules obtained by the direct image functor along the orbits, the affinity of the imbedding maps which was proved by Beilinson--Bernstein is also important in technical aspects (\cite[4.3. Theorem, Setion A.3.3, 4.1. Proposition]{MR910203}).
	
	Matsuki and Richardson--Springer studied combinatorial aspects of the classification of the orbits. In particular, they proved the following result:
	
	\begin{thm}[{\cite[Example 1, Theorem 3]{MR527548}, \cite[2.7 Proposition (i), 2.8 Proposition]{MR1066573}}]\label{thm:matsuki}
		Let $G$ be a connected reductive algebraic group over an algebraically field $F$ of characteristic not two, equipped with an involution $\theta$. Let $K\subset G$ be the $\theta$-fixed point subgroup. Let $\cB_G$ be the flag variety of $G$. Fix a system $\cT$ of complete representatives of $K$-conjugacy classes of $\theta$-stable maximal tori of $G$. For each $H\in\cT$, we let $W(\Delta(G,H))$ be the Weyl group of $(G,H)$. Let $W_K(\Delta(G,H))$ be the subgroup of $W(\Delta(G,H))$ consisting of elements represented by $F$-points of the normalizer $N_K(H)$ of $H$ in $K$. Choose any system $W(\Delta(G,H))'$ of complete representatives of $W_K(\Delta(G,H))\backslash W(\Delta(G,H))$. Set $V=\coprod_{H\in\cT} W(\Delta(G,H))'$. For each $H\in\cT$, fix a Borel subgroup $B_H$ containing $H$. For each $v\in W(\Delta(G,H))'\subset V$, let $\cO_v$ be the $K$-orbit in $\cB_G$ attached to the twist of $B_H$ by $v$. Then we have $\cB_G(F)=\coprod_{v\in V} \cO_v(F)$.
	\end{thm}

	Recently, studies on $(\fg,K)$-modules over commutative rings have been developed by authors. In particular, Harris proposed in \cite[Theorem 2.1.3]{MR3053412} to obtain rational forms of irreducible Harish-Chandra modules through certain equivariant twisted D-modules on the flag variety by taking forms of the construction in \cite[Section 3]{MR610137} as their geometric aspects. This was amended by \cite{MR4073199} for rationality patterns of their geometric ingredients including the orbits in the flag variety. Subsequently, the theory of twisted D-modules over schemes was developed in \cite[Section 1-3]{hayashijanuszewski} to provide a general formalism of a geometric construction of $(\fg,K)$-modules over commutative rings (schemes).
	
	For integrality patterns of orbits of $\theta$-stable parabolic subgroups, we proved in \cite{hayashijanuszewski} that for a reductive group scheme $G$ over a $\bZ[1/2]$-scheme $S$ with an involution $\theta$, the \'etale quotient of the moduli scheme $\cP^\theta_G$ of $\theta$-stable parabolic subgroups of $G$ by the $\theta$-fixed point subgroup scheme $K\subset G$ is represented by a finite \'etale $S$-scheme if $\pi_0(K)$ is so (\cite[Theorem 4.1.7]{hayashijanuszewski}). This was achieved by establishing an \'etale local combinatorial classification of geometric $K$-conjugacy classes of $\theta$-stable parabolic subgroups. In particular, if $S=\Spec F$ for an algebraically closed field $F$ of characteristic not two, the quotient morphism $rt:\cP^\theta_G\to K\backslash\cP^\theta_G$ classifies the $K$-conjugacy classes of $\theta$-stable parabolic subgroups, i.e., the fibers of $rt$ are the $K$-orbits in $\cP^\theta_G$. As an application, we achieved a geometric construction of arithmetic forms of cohomologically induced modules and the descent of their rings of definition simultaneously (\cite[Section 5]{hayashijanuszewski}).
	
	The advantage of this sheaf-theoretic approach is that we can guarantee the existence of the $S$-form of the orbit decomposition without discussing the explicit descent datum on orbit decompositions over \'etale loci of $S$. We also remark that due to the descent, the $S$-forms of the orbits may not have base points. In particular, they are not orbits over $S$. See \cite[Section 0.4]{hayashijanuszewski} for an elementary example.
	
	The aim of this paper is to find a finite \'etale $S$-scheme which \'etale locally parameterizes $K$-orbits in the flag scheme $\cB_G$ (= the moduli scheme of Borel subgroups of $G$) and knows the base of each local $K$-orbit by giving a $K$-orbit decomposition \'etale locally in the fashion of Theorem \ref{thm:matsuki} as an analog of \cite[Theorem 4.1.7]{hayashijanuszewski} and then by taking the quotient.
	
	Motivated by the appearance of (complete representatives of) all the $K$-conjugacy classes of $\theta$-stable maximal tori in Theorem \ref{thm:matsuki}, we start with an analog of \cite[Theorem 4.1.7]{hayashijanuszewski} for $\theta$-stable maximal tori:
	
	\begin{thm}[{Section \ref{sec:tor}}]\label{thm:tor}
		Let $G$ be a reductive group scheme over a $\bZ[1/2]$-scheme $S$, equipped with an involution $\theta$. Let $K\subset G$ be the $\theta$-fixed point subgroup scheme.
		\begin{enumerate}
			\item The fppf quotient of the moduli scheme $\Tor_G^\theta$ of $\theta$-stable maximal tori of $G$ by $K$ is represented by a finite \'etale $S$-scheme.
			\item If $S=\Spec F$ for an algebraically closed field $F$ of characteristic not two, the quotient map $\Tor_G^\theta\to K\backslash\Tor_G^\theta$ classifies the $K$-conjugacy classes of $\theta$-stable maximal tori.
		\end{enumerate}
	\end{thm}
	
	We prove this by lifting corresponding results at geometric points (see \cite[Section 9]{MR1066573}) locally in the \'etale topology of $S$. In fact, we locally find $\theta$-stable maximal tori, of which $\Tor_G^\theta$ is decomposed into the $K$-orbits by the disjoint union as an $S$-scheme (to be precise $S'$-scheme for an \'etale $S$-scheme $S$).
	
	One might expect that we could prove a similar result for the whole of the flag scheme by following Theorem \ref{thm:matsuki} as is. For test cases, we gave set-theoretic decomposition of the flag scheme for $\SL_2$ and $\SL_3$ (\cite[Section 0.4]{hayashijanuszewski} and \cite[Theorem 1.1]{hayashikgb}). However, our expectation fails even in these cases because of nontrivial closure relations among $K$-orbits in the flag scheme. In fact, the quotient of the flag scheme by $K$ is not representable in general:
	
	\begin{ex}
		Put $S=\Spec\bC$, where $\bC$ is the field of complex numbers. Set $G=\SL_2$ and $\theta=((-)^T)^{-1}$, where $(-)^T$ is the transpose map. Let $\cB_G$ be the flag variety of $G$. Then $\theta$ induces an involution on $\cB_G$, and we denote its fixed point subvariety by $\cB^\theta_G$. Then we have a monomorphism
		$K\backslash\cB^\theta_G\hookrightarrow K\backslash\cB_G$.
		One can easily show $\Spec\bC\coprod \Spec\bC\cong K\backslash\cB^\theta_G$.
		
		If $K\backslash\cB_G$ is representable, this is also a categorical quotient. It however turns out that $K\backslash\cB_G\cong S$ since the unique open $K$-orbit is scheme-theoretically dense in $\cB_G$. It is evident that the canonical morphism $\Spec\bC\coprod \Spec\bC\to \Spec\bC$ is not a monomorphism. This shows that $K\backslash\cB_G$ is not representable. Though the categorical quotient exists, this is not what we want in the sense that it collapses information on the orbit decomposition due to the nontrivial closure relations.
	\end{ex}
	
	We realize from this observation that the decomposition of Theorem \ref{thm:matsuki} is not as a variety but as a set. Therefore we cannot obtain a finite \'etale $S$-scheme representing ``the structure of the fppf local $K$-orbit decomposition'' by the categorical or sheaf-theoretic quotient.
	
	In the situation of the Bruhat decomposition, this issue was resolved in \cite{MR0218364} by introducing the notion of standard position for pairs of parabolic subgroups to define a bijective subobject of the (double) total flag scheme. For example, if we are given a Borel subgroup $B$, the quotient of the moduli scheme $\cB^{B\mathrm{-}\std}_{G}$ ($=\underline{\mathrm{Par}}_\emptyset(G;B)$ in the notation of \cite[Section 4.5.4]{MR0218364}) of Borel subgroups which are at standard position with $B$ is represented by a finite \'etale $S$-scheme (\cite[Sections 4.5.4, 4.5.5]{MR0218364}). In fact, $\cB^{B\mathrm{-}\std}_{G}$ is \'etale locally the disjoint union of $B$-orbits (the Bruhat cells) (\cite[Section 4.5.5]{MR0218364}).
	
	Following this idea, let us introduce:
	
	\begin{defn}[cf.~{\cite[Definition 4.1.6]{hayashijanuszewski}}]
		\begin{enumerate}
			\item Let $\cB^{\theta\mathrm{-}\std}_G$ be the moduli space of Borel subgroups $B$ which are at standard position with $\theta(B)$ (see \eqref{eq:defn} for details).
			\item Let $rt:\cB^{\theta\mathrm{-}\std}_G\to K\backslash \cB^{\theta\mathrm{-}\std}_G\reflectbox{$\coloneqq$}\rtype_\emptyset(G,\theta)$ denote the fppf quotient map.
			\item For $v\in \rtype_\emptyset(G,\theta)(S)$, set $\cB^{\theta\mathrm{-}\std}_{G,v}\coloneqq rt^{-1}(v)$.
			Let $i_v:\cB^{\theta\mathrm{-}\std}_{G,v}\hookrightarrow \cB_G$ denote the inclusion map.
		\end{enumerate}
	\end{defn}
	
	The $S$-space $\cB^{\theta\mathrm{-}\std}_G$ is expected to be the $S$-form of the disjoint union of local $K$-orbits. The quotient $\rtype_\emptyset(G,\theta)$ is expected to classify local $K$-orbits and know their local bases. Then we obtain $S$-forms of local $K$-orbits by (3). In this paper, we study these objects to discuss that our expectations are true within the framework of the theory of schemes.
	
	The easy part is that $\cB^{\theta\mathrm{-}\std}_G$ is represented by a smooth $S$-scheme (Proposition \ref{prop:smooth}). Our wish is to prove that $\rtype_\emptyset(G,\theta)$ is represented by a finite \'etale $S$-scheme. We also wish to guarantee that $rt$ locally classifies the $K$-orbits. We plan to achieve these by lifting the combinatorial classification of Theorem \ref{thm:matsuki}. For this, let us introduce a scheme-theoretic analog of $W_K(\Delta(G,H))$: For a $\theta$-stable maximal torus $H$ of $G$, set $W_K(G,H)\coloneqq N_{K}(H)/(H\cap K)$,
	where $N_{K}(H)$ is the normalizer of $H$ in $K$ (see the notation section below if necessary). The quotient is taken in the fppf topology. Towards the representability by a finite \'etale $S$-scheme, let us think of:
	
	\begin{ass}[Local constancy hypothesis]\label{ass}
		For all $\theta\times_S S'$-stable maximal tori $H'$ over \'etale $S$-schemes $S'$, $W_{K\times_S S'}(G\times_S S',H')$ are finite \'etale group $S'$-schemes.
	\end{ass}
	
	\begin{rem}\label{rem:closed}
		For any $\theta$-stable maximal torus $H$ of $G$, $W_K(G,H)$ is an affine open subgroup scheme of the Weyl group scheme $W(G,H)$ (Corollary \ref{cor:opensubgroup}). Therefore $W_K(G,H)$ is represented by a finite \'etale group $S$-scheme if and only if it is a closed subgroup scheme of $W(G,H)$. An expected way to verify that $W_K(G,H)$ is finite \'etale is to find a base-free combinatorial description of its geometric fibers possibly after passage to fppf localization. Though it should be expected in the standard examples of \cite{MR4627704}, generalities may not exist (cf.~\cite[2.10 Remarks]{MR1066573}).
	\end{rem}

	\begin{ex}[Remark \ref{rem:localconstancy}]\label{ex:field}
		The local constancy hypothesis holds if $S=\Spec F$ for a field $F$ of characteristic not two.
	\end{ex}
	
	We are now ready to state the main results of this paper:
	
	\begin{thm}[Sections \ref{sec:thm1,2} and \ref{sec:bb}]\label{mainthm}
		Consider the same setting as Theorem \ref{thm:tor}.
		\begin{enumerate}
			\item If the local constancy hypothesis hods, the fppf sheaf $\rtype_\emptyset(G,\theta)$ is represented by a finite \'etale $S$-scheme.
			\item If $S=\Spec F$ for an algebraically closed field $F$ of characteristic not two, $rt$ classifies the $K$-conjugacy classes of Borel subgroups (recall Example \ref{ex:field} and (1)).
			\item Assume that the fppf sheaf $\rtype_\emptyset(G,\theta)$ is represented by a finite \'etale $S$-scheme. Then for every $S$-point $v$ of $\rtype_\emptyset(G,\theta)$, $i_v$ is an affine immersion.
		\end{enumerate} 
	\end{thm}
	
	Part (1) is the scheme-theoretic counterpart of Theorem \ref{thm:matsuki}. For the proof, we plan to lift Theorem \ref{thm:matsuki} to establish an orbit decomposition \'etale locally. We will use the proof of Theorem \ref{thm:tor} (particularly, local complete system of representatives of geometric $K$-conjugacy classes of $\theta$-stable maximal tori) and the local constancy hypothesis to verify that $\cB^{\theta\mathrm{-}\std}_G$ is \'etale locally ``closure relation free'' (Lemma \ref{lem:transB}). Then we take the quotient to win. We note that (1) generalizes \cite[Theorem 4.1.7]{hayashijanuszewski} inside the flag scheme for symmetric subgroups and \cite[Theorem 1.1]{hayashikgb}. One can readily see (2) from the proof of (1). Part (3) is the scheme-theoretic counterpart of Beilinson--Bernstein's affinity theorem. This is verified by lifting the proof of Beilinson--Bernstein's argument. In fact, they divided the flag variety into affinely imbedded $K$-invariant subvarieties which are described as fibers of the Springer map $\varphi$ in \cite{MR803346}\footnote{In his original definition, Springer took the quotient by $K$ for the domain. Since we do not know that the quotient is representable on the course of our proof, we do not take the quotient for the definition of $\varphi$.}. Then they proved that the $K$-orbits inside them are open and closed. In our proof, we divide $\cB^{\theta\mathrm{-}\std}_G$ scheme-theoretically into fibers of the scheme-theoretic enhancement of $\varphi$ which we denote by the same symbol. Under the assumption of (3), fibers of $\varphi$ are disjoint unions of $K$-orbits since $\varphi$ is a $K$-invariant morphism to a finite \'etale $S$-scheme. In particular, the orbits are open and closed in the fibers of $\varphi$.
	
	\begin{rem}
		In the original proof of the affinity in \cite{MR910203} after Beilinson--Bernstein, $\varphi$ did not appear apparently. In fact, their and our parameterizations of the decomposition at the first stage of the proof are equivalent but slightly different. Their relation is explained in Remark \ref{rem:comparison}. The Springer map appeared in \cite{MR803346} for study of $K$-orbits in the flag variety. Then Richardson--Springer promoted the ideas in \cite{MR803346} to give further combinatorial results on the $K$-orbit decomposition of the flag variety. We adopt the use of $\varphi$ for direct applications of then (see below).
	\end{rem}

	The remaining difficulty in applications to examples is to check the local constancy hypothesis as mentioned in Remark \ref{rem:closed}. Inspired by the proof of (3), we can prove the assertions in Theorem \ref{mainthm} without the hypothesis but under a stronger assumption at geometric points. To formulate it, let us note an easy fact that $\varphi$ is smooth (Lemma \ref{lem:varphi_smooth}). Hence the image of $\varphi$ is endowed with the structure of an open subscheme of the target of $\varphi$, which we denote by $\cI'_{G,\theta}$. In the statement below, for a geometric point $\bar{s}$ of $S$, let $(-)_{\bar{s}}$ denote the geometric fiber at $\bar{s}$.
	
	\begin{var}[Section \ref{sec:var}]\label{var:mainthm}
		Consider the same setting as Theorem \ref{thm:tor}. 
		\begin{enumerate}
			\item The $S$-scheme $\cI'_{G,\theta}$ is finite \'etale.
			\item Assume that for each geometric point $\bar{s}$ of $S$, the fiber of $\varphi_{\bar{s}}$ at every point of $\cI'_{G,\theta,\bar{s}}$ consists of a single $K_{\bar{s}}$-orbit. Then the morphism $\varphi$ gives rise to an isomorphism from the fppf quotient sheaf $\rtype_\emptyset(G,\theta)$ onto $\cI'_{G,\theta}$. In particular, $\rtype_\emptyset(G,\theta)$ is represented by a finite \'etale $S$-scheme.
		\end{enumerate}
	\end{var}
	
	In this setting, we can skip the latter part of the former proof of (3) since \'etale locally the fibers of $\varphi$ give a $K$-orbit decomposition of $\cB^{\theta\mathrm{-}\std}_G$. We prove (1) by lifting the combinatorial description of $\cI'_{G,\theta,\bar{s}}$ in \cite[7.13 Theorem]{MR1066573}. The basic idea for (2) is as follows: The assumption says that $\varphi$ geometric-fiberwisely gives the orbit decomposition by the fibers of $\varphi$. We guarantee its local constancy by (1). As a result, we \'etale locally obtain a decomposition of $\cB^{\theta\mathrm{-}\std}_G$ into open subschemes, each of which consists of a single orbit at every geometric point. We lift base points \'etale locally to express them as orbits.

	\begin{rem}
		Since the formations of taking open subschemes respect smooth schemes over a fixed base $S$, we can readily see that the theorems stated above hold if we replace $K$ with any open and closed subgroup scheme (also replace $K$ in Assumption \ref{ass} accordingly). Then we should replace $\rtype_\emptyset(G,\theta)$ with $\rtype_\emptyset(G,\theta,K)$ for the symbol of the fppf quotient $K\backslash\cB^{\theta\mathrm{-}\std}_G$. On the other hand, $K\backslash\Tor_G^\theta$ in Theorem \ref{thm:tor} is independent of this replacement from its proof. In our discussions below, we only treat the fixed point subgroup scheme for simplicity.
	\end{rem}

	At the end of this paper, we see some concrete examples (Section \ref{sec:ex}). In particular, we confirm that Theorem \ref{mainthm} and Variant \ref{var:mainthm} are scheme-theoretic consequences of the ground algebraically closed field free feature of the orbit classifications in \cite[Section 10]{MR1066573}. We also note that working over smaller base rings is related to the descent phenomena of orbits studied in \cite[Theorem 4.1.7]{hayashijanuszewski} and \cite[Theorem 1.1]{hayashikgb}.
	
	We would like to close this section with an application to representation theory. As a consequence of the present work, we can generalize the construction of the standard Harish-Chandra sheaves in \cite[Section 2]{MR910203} by using the theory of twisted D-modules over schemes (\cite[Theorems 3.9.2 and 3.10.1]{hayashijanuszewski}):

	\begin{cons}[Standard Harish-Chandra sheaf]
		Let $G$ be a reductive group scheme over a commutative Noetherian $\bZ[1/2]$-algebra of finite Krull dimension, equipped with an involution $\theta$. Write $\fg$ for the Lie algebra of $\fg$. Let $K\subset G$ be an open and closed subcheme of the $\theta$-fixed point subgroup scheme. Let $\cA$ be a $G$-equivariant tdo on $\cB_G$ (\cite[Sections 1.1, 1.4]{hayashijanuszewski}).
		
		Suppose that the fppf quotient sheaf $\rtype_\emptyset(G,\theta,K)=K\backslash\cB^{\theta\mathrm{-}\std}_G$ is represented by a finite \'etale $k$-scheme. Let $v\in \rtype_\emptyset(G,\theta,K)(k)$. Let $\cM$ be a $K$-equivariant quasi-coherent left $i^\cdot_v\cA$-module (see \cite[Section 1.3.3]{hayashijanuszewski} for definitions). Then we obtain a $K$-equivariant quasi-coherent left $\cA$-module $(i_v)_+\cM$ (\cite[Section 3.7, Theorem 3.9.2]{hayashijanuszewski}). Take the global sections to obtain a $(\fg,K)$-module $\Gamma(\cB_G,(i_v)_+\cM)$ in the sense of \cite{MR4007195} by \cite[Theorem 3.10.1]{hayashijanuszewski}.
	\end{cons}
	
	The finiteness assumption on $k$ is for a nice behavior of the direct image functor $i_+$ (\cite[Theorem 3.9.2, Example A.7.4]{hayashijanuszewski}). We deduce the following assertion from \cite[Theorems 3.7.30 and 3.10.1]{hayashijanuszewski}:
	
	\begin{thm}
		The constructions $(i_v)_+\cM$ and $\Gamma(\cB_G,(i_v)_+\cM)$ commute with flat base changes.
	\end{thm}
	
	In particular, we obtain $k$-forms of the standard modules (the dual of cohomologically induced modules) arising from non-closed orbits by \cite[4.3. Theorem]{MR910203}. To be precise, if $k$ is a Noetherian subring of $\bC$ of finite Krull dimension and the base change of $\cM$ to $\bC$ is irreducible, $\Gamma(\cB_G,(i_v)_+\cM)\otimes_k\bC$ is a standard module.

	\renewcommand{\abstractname}{Acknowledgments}
	\begin{abstract}
		This work was supported by JSPS KAKENHI Grant Number JP22KJ2045.
	\end{abstract}

	\section*{Convention}
	In this paper, quotient by group schemes is taken in the fppf topology.
	
	\section*{Notation}
	
	Let $\bZ$ be the ring of integers. Let $\bQ$ (resp.~$\bR$, $\bC$) denote the field of rational (resp.~real, complex) numbers.
	
	For a group, we express its unit by $e$.
	
	Let $S$ be a scheme, and $s\in S$. We denote the residue field of $S$ at $s$ by $\kappa(s)$.
	We write $\bar{s}$ for the geometric point attached to $s$, i.e., the naturally induced map from the spectrum of the algebraic closure of $\kappa(s)$ to $S$. More generally, geometric points of $S$ and their corresponding algebraically closed fields will be denoted by $\bar{s}$ and $\kappa(\bar{s})$ respectively. For an $S$-scheme $X$, we denote its fiber at a point $s\in S$ (resp.~a geometric point $\bar{s}$ of $S$) by $X_s$ (resp.~$X_{\bar{s}}$). We apply a similar notation to morphisms as well. For an $S$-scheme $T$, we write $X(T)$ for the set of $T$-points of $X$. If $T=\Spec A$ for a commutative ring $A$, then we sometimes write $X(A)=X(T)$.
	
	For a torus $T$ over an algebraically closed field, we denote its cocharacter group by $X_\ast(T)$. We set $E(T)=\bR\otimes_\bZ X_\ast(T)$.
	
	For a finite set $E$ and a scheme $S$, let $E_S$ denote the attached constant $S$-scheme (i.e., the disjoint union of copies of $S$ indexed by $E$). Similarly, if we are given a map $f:E\to E'$ of finite sets, we write $f_S$ for the induced constant morphism $E_S\to E'_S$.
	
	Let $S$ be a scheme, $K\hookrightarrow G$ be a monomorphism of group $S$-schemes, $X,X'\rightrightarrows G$ be monomorphisms of $S$-schemes. Then we define presheaves $\Transp_K(X,X')$ and $\STransp_K(X,X')$ on the category of $S$-schemes by
	\[\begin{array}{cc}
		T\mapsto\{k\in K(T):~k Xk^{-1}\subset X'\},
		&T\mapsto\{k\in K(T):~k Xk^{-1}= X'\}
	\end{array}\]
	(\cite[D\'efinition 6.1]{MR0234961}). If $X=X'$, we write
	$N_K(X)\coloneqq\STransp_K(X,X)$.
	This is equal to $\Transp_K(X,X)$ if $X$ is of finite presentation (\cite[Proposition (17.9.6)]{MR0238860}).
	
	Let $S$ be a scheme. Let $X$ be an $S$-scheme, equipped with an action of a group scheme $K$ over $S$. Let $x\in X(S)$. Then we denote the stabilizer subgroup of $K$ at $x$ by $Z_K(x)$ (\cite[Scholie 2.3.3.1]{MR0212028}). Set $Kx=K/Z_K(x)$, which will be referred to as the $K$-orbit attached to $x$. We call the canonical monomorphism $Kx\hookrightarrow X$ the orbit map in this paper.
	
	Let $G$ be a reductive group scheme over a scheme $S$. If we are given a splitting of $G$ with a split maximal torus $H$ (\cite[D\'efinition 1.13]{MR0218363}), we denote the set of roots (resp.~the Weyl group of the corresponding root datum) by $\Delta(G,H)$ (resp.~$W(\Delta(G,H))$). For a maximal torus $H$ of $G$, the Weyl group scheme of $(G,H)$ will be written as $W(G,H)$ (\cite[Section 3.1]{MR0218363}). If a section of $W(G,H)$ admits a lift to the normalizer of $H$ in $G$, we denote it by the same symbol. We remark that the lift exists if $(G,H)$ admits a splitting (\cite[Corollaire 3.8]{MR0218363}). Moreover, $W(G,H)$ is isomorphic to the constant group $S$-scheme $W(\Delta(G,H))_S$ (\cite[Proposition 3.4]{MR0218363}). Under this identification, we apply the same notation to elements of the Weyl group by regarding them as constant sections. We write $\cB_G$ and $\Tor_G$ for the moduli $S$-schemes of Borel subgroups and maximal tori respectively. We write $\Kill_G$ for the Killing scheme, i.e., the moduli $S$-scheme of pairs of Borel subgroups and their maximal tori. See \cite[Corollaire 5.8.3]{MR0218363} for the representability of these three spaces. We denote the moduli scheme of pairs of Borel subgroups of standard position as $\Stand_{G,\emptyset,\emptyset}$. Throughout this paper, we put the diagonal $G$-action on $\Stand_{G,\emptyset,\emptyset}$. The quotient
	$\StandType_{G,\emptyset,\emptyset}\coloneqq G\backslash\Stand_{G,\emptyset,\emptyset}$
	is represented by a finite \'etale $S$-scheme. The quotient map $\Stand_{G,\emptyset,\emptyset}\to \StandType_{G,\emptyset,\emptyset}$ will be denoted by $t_2$.
	See \cite[Section 4.5.3]{MR0218364} for details.	
	
	For $n\geq 1$, we write $\fS_n$ for the $n$th symmetric group.
	
	For square matrices $A_1,A_2,\ldots,A_n$, we denote the block diagonal matrix
	\[\left(\begin{array}{cccc}
		A_1 &  &  &  \\
		& A_2 &  &  \\
		&  & \ddots &  \\
		&  &  & A_n
	\end{array}
	\right)\]
	by $\diag(A_1,A_2,\ldots,A_n)$. We apply similar notations to block diagonal group schemes of matrices. For a nonnegative integer, let $I_n$ denote the unit matrix of size $n$.
	
	\section{Preliminaries}

	\subsection{Fixed point subscheme}\label{sec:inv}
	
	Following \cite[Lemma 3.1.1]{MR4627704}, we define the invariant subscheme for a scheme with an involution. To be precise, it was done for affine bases. We remark that the representability and the closedness over general base schemes are verified through the literally the same argument. As for the smoothness, we may assume the base to be affine since the statement is local on the base. Let us record this argument as a statement:
	
	\begin{defn-prop}\label{defn-prop:x^theta}
		Let $S$ be a scheme, and $X$ be an $S$-scheme with an involution $\theta$.
		\begin{enumerate}
			\item The presheaf over the category of $S$-schemes defined by
			\[T\mapsto \{x\in X(T):~\theta(x)=x\}\]
			is represented by a subscheme of $X$, which we will denote by $X^\theta$.
			\item If $X$ is locally of finite presentation over $S$, so is $X^\theta$. In particular, the inclusion map $X^\theta\hookrightarrow X$ is then locally of finite presentation.
			\item If $X$ is separated over $S$, $X^\theta$ is a closed subscheme of $X$.
			\item Assume that $S$ is a $\bZ[1/2]$-scheme. If $X$ is smooth and separated over $S$, so is $X^\theta$.
		\end{enumerate}
	\end{defn-prop}

	For a torus $H$ over a $\bZ[1/2]$-scheme $S$, equipped with an involution $\theta$, we define a subtorus $H^-$ as follows: Since $H$ is commutative, $h\mapsto \theta(h)^{-1}$ is an involution of $H$. Its invariant closed subgroup scheme is smooth of multiplicative type by Definition-Proposition \ref{defn-prop:x^theta} and \cite[Corollary B.3.3]{MR3362641}. We set $H^-$ as its unit component. This is a torus by definition.

	\subsection{Orbits}
	
	In this section, we collect basic generalities on representability and geometric properties of orbits. Let $S$ be a scheme, and $K$ be a smooth quasi-compact and quasi-separated group scheme over $S$.

	\begin{lem}[{\cite[Lemma 4.1]{hayashicontraction}}]\label{lem:quotient}
		Let $L$ be a closed subgroup scheme of $K$ which is flat and locally of finite presentation (e.g.~smooth) over $S$. Then $K/L$ is a smooth quasi-compact separated algebraic space over $S$.
	\end{lem}

	\begin{prop}\label{prop:closed}
		Let $X$ be an $S$-scheme, equipped with an action of $K$. Pick a section $x\in X(S)$.
		\begin{enumerate}
			\item If $X$ is separated over $S$, $Z_K(x)$ is a closed subgroup scheme of $K$.
			\item If $X$ is locally of finite presentation over $S$, and $Z_K(x)$ is flat over $S$ then $Kx$ is represented by an $S$-scheme.
		\end{enumerate}
	\end{prop}
	
	\begin{proof}
		See \cite[Exemples 6.2.4.~b)]{MR0234961} (resp.~\cite[Th\'eor\`eme 10.1.2]{MR0257095})
		for (1) (resp.~(2)).
	\end{proof}
	
	As an application, let us record:
	
	\begin{cor}\label{cor:opensubgroup}
		Consider the setting of Theorem \ref{thm:tor}. Let $H$ be a $\theta$-stable maximal torus of $G$. Then $W_K(G,H)$ is represented by an affine open subgroup scheme of $W(G,H)$.
	\end{cor}
	
	\begin{proof}
		It is evident that $N_G(H)$ is $\theta$-stable in $G$. Let us also recall that $N_G(H)$ is a smooth closed subgroup scheme of $G$ (\cite[Proposition 2.1.2]{MR3362641}). Definition-Proposition \ref{defn-prop:x^theta} thus implies that $N_K(H)=N_G(H)^\theta$ and $H\cap K=H^\theta$ are smooth closed subgroup schemes of $G$. Observe that the kernel of the composite map
		\begin{equation}
			N_K(H)\hookrightarrow N_{G}(H)\to W(G,H)\label{eq:comp}
		\end{equation}
		of the canonical inclusion and quotient maps is $H\cap K$. Apply \cite[Corollaire 10.1.3]{MR0257095} to \eqref{eq:comp} to deduce that $W_K(G,H)$ is representable. Lemma \ref{lem:quotient} also implies that $W_K(G,H)$ is smooth over $S$. The monomorphism $W_K(G,H)\hookrightarrow W(G,H)$ is an open immersion since it is smooth (\cite[Proposition 2.4.1 (iv)]{MR2791606}, \cite[Th\'eor\`eme (17.9.1)]{MR0238860}).
		
		Finally, we prove that $W_K(G,H)$ is affine. Observe that $H\cap K$ is of multiplicative type by \cite[Corollary B.3]{MR3362641}. Hence to see that $W_K(G,H)$ is affine, we may work locally in the fppf topology of $S$ to assume that $H\cap K$ is diagonalizable. In this case, the assertion follows from \cite[Th\'eor\`eme 5.1]{MR0212024}. In fact, note that the fppf quotient $W_K(G,H)=N_K(H)/(K\cap H)$ is also the fpqc quotient since $W_K(G,H)$ is represented by an $S$-scheme and in particular it is a sheaf in the fpqc topology.
	\end{proof}
	
	\begin{ex}\label{ex:fun}
		Suppose that $H$ is a fundamental Cartan subgroup, i.e., the centralizer of a maximal torus $T$ of $K$ (\cite[Theorem 3.1.3, Example 3.1.2]{MR4627704}\footnote{For the representability over general bases, see \cite[Corollaire 5.3]{MR0215856} or \cite[Lemma 2.2.4]{MR3362641}. Then work locally in the Zariski topology to deduce the version of \cite[Theorem 3.1.3]{MR4627704} for general base schemes with $1/2$.}). Then we have a canonical isomorphism $N_K(T)/(H\cap K)\cong W_K(G,H)$. In particular, if $K$ has connected geometric fibers, $W_K(G,H)\cong W(K,T)$ and $W_K(G,H)$ is finite \'etale over $S$.
	\end{ex}
	
	\subsection{Hensel's lemma}
	
	A basic idea for the proofs of the representability theorems and Variant \ref{var:mainthm} (2) in Section \ref{sec:intro} is to lift the corresponding results over algebraically closed fields of characteristic not two. In SGA 3, a general technique for this kind of idea was developed as Hensel's lemma of \cite[Proposition 1.10]{MR0215856}. Let us record its version suited to our applications in this paper:
	
	\begin{prop}\label{prop:hensel}
		Every smooth morphism \'etale locally admits a section. That is, let $f:X\to S$ be a smooth morphism of schemes. Then there exist an \'etale covering $\{S_\lambda\to S\}$ such that $(X\times_S S_\lambda)(S_\lambda)\neq \emptyset$ for every $\lambda$, where $X\times_S S_\lambda$ is regarded as an $S_\lambda$-scheme in a natural way.
	\end{prop}
	
	\begin{proof}
		Pick $s\in S$. Since $X_s$ is nonempty and smooth over $\kappa(s)$, one can find $x_s\in X_s$ with $\kappa(x)$ finite separable over $\kappa(s)$. Apply \cite[Proposition 1.10]{MR0215856} to $x_s$ to obtain an \'etale morphism $S'\to S$ hitting $s$ and an $S$-morphism $x:S'\to X$. It gives rise to an element of $(X\times_S S')(S')$. Run through all $s$ to deduce the assertion.
	\end{proof}

	Towards the proof of Theorem \ref{mainthm} (1), we also need to analyze the moduli space $\cB^{\theta\mathrm{-}\std}_G$. The key observation will be to verify that $\cB^{\theta\mathrm{-}\std}_G$ is \'etale locally ``closure relation free''. That is, every section of $\cB^{\theta\mathrm{-}\std}_G$ \'etale locally arises from a $K$-orbit over an \'etale $S$-scheme. To do this, we wish to lift an orbit at a geometric point $\bar{s}$ to a $K$-orbit around an \'etale neighborhood of $\bar{s}$. On the course, we will also have to manage a similar issue for tori $H^-$ due to the combinatorial classification of $\theta$-stable maximal tori in \cite[Section 9]{MR1066573} (see Section \ref{sec:richardson--springer} for a brief exposition). The assertion below obtained by a simple application of \cite[Proposition 1.10]{MR0215856} will resolve these issues:
	
	\begin{prop}\label{prop:transp}
		Let $S$ be a scheme, $K\hookrightarrow G$ be a monomorphism of group $S$-schemes, and $Y,Y'\rightrightarrows G$ be monomorphisms of $S$-schemes. Assume that the presheaf $\Transp_K(Y,Y')$ is represented by a smooth $S$-scheme. Then the following conditions are equivalent:
		\begin{enumerate}
			\renewcommand{\labelenumi}{(\alph{enumi})}
			\item There are an \'etale covering $\{S_\lambda\to S\}$ and $k_\lambda\in K(S_\lambda)$ such that
			\[k_\lambda (Y\times_S S_\lambda)k^{-1}_\lambda\subset Y'\times_S S_\lambda;\]
			\item There are an fpqc covering $\{S_\lambda\to S\}$ and $k_\lambda\in K(S_\lambda)$ such that
			\[k_\lambda (Y\times_S S_\lambda)k^{-1}_\lambda\subset Y'\times_S S_\lambda;\] 
			\item For every geometric point $\bar{s}$, there exists $k\in K(\kappa(\bar{s}))$ such that $k Y_{\bar{s}} k^{-1}\subset Y'_{\bar{s}}$;
			\item The structure morphism $\Transp_K(Y,Y')\to S$ is surjective.
		\end{enumerate}
		A similar assertion holds for $\STransp$.
	\end{prop}
	
	\begin{proof}
		This is proved in a similar way to \cite[Corollaire 5.4]{MR0215856}. 
	\end{proof}
	
	We would like to record two situations where $\Transp_K$ and $\STransp_K$ are smooth. Let $S$ be a $\bZ[1/2]$-scheme. Let $G$ be a smooth $S$-affine group scheme over $S$, equipped with an involution $\theta$. Set $K=G^\theta$.
	
	\begin{prop}\label{prop:multiplicative}
		If $T,T'\subset G$ are $\theta$-stable closed subgroup schemes of multiplicative type, $\Transp_K(T,T')$ is smooth.
	\end{prop}
	
	\begin{proof}
		According to \cite[Proposition 2.1.2]{MR3362641}, $\Transp_G(T,T')$ is a $\theta$-stable smooth closed subscheme of $G$. The assertion now follows from Definition-Proposition \ref{defn-prop:x^theta}.
	\end{proof}

	\begin{prop}\label{prop:type(R)}
		Assume that $G$ is reductive.
		For $\theta$-stable subgroup schemes $T,T'\subset G$ of type (R) in the sense of \cite[D\'efinition 5.2.1]{MR0218363}, $\STransp_K(T,T')$ is smooth over $S$.		
	\end{prop}
	
	\begin{proof}
		This is an immediate consequence of \cite[Th\`eor\'eme 5.3.9]{MR0218363} and Definition-Proposition \ref{defn-prop:x^theta}.
	\end{proof}
	
	Note that we will also use the original form \cite[Proposition 1.10]{MR0215856} of Hensel's lemma later. See the paragraph below Proposition \ref{prop:smooth}.
	
	\section{Geometric $K$-conjugacy classes of $\theta$-stable maximal tori}
	
	In this section, we aim to prove Theorem \ref{thm:tor}. We recall that a combinatorial description of the set of $K$-conjugacy classes of $\theta$-stable maximal tori over an algebraically closed field of characteristic not two in terms of root systems was given in \cite[Section 9]{MR1066573}. Our strategy is as follows: We regard it as the set-theoretic result at geometric points. We lift it onto their \'etale neighborhoods to achieve a scheme-theoretic orbit decomposition \'etale locally. Finally, we take the quotient to finish the proof.
	
	In this section, let $S,G,\theta,K$ be as in Theorem \ref{thm:tor}.
	
	\subsection{Richardson--Springer's classification}\label{sec:richardson--springer}
	
	In this section, we review \cite[Section 9]{MR1066573}. Throughout this section, we put $S=\Spec F$, where $F$ is an algebraically closed field of characteristic not two. We pick an $F$-point $B_0$ of the unique open $K$-orbit in $\cB_G$. We choose a $\theta$-stable maximal torus $H_0$ of $B_0$. In a sequel, we identify $E(H^-_0)$ with a subspace of $E(H_0)$ for the canonical injective map $E(H^-_0)\hookrightarrow E(H_0)$ (back to Section \ref{sec:inv} for the definition of $H^-_0$). Choose any inner product $(-,-)$ on $E(H_0)$ which is invariant under the Weyl group of $(G,H_0)$ and $\theta$ (\cite[\S4]{MR0656625}). Let $\Delta(G,H_0)$ be the set of roots of $(G,H_0)$, which we regard as a subset of $E(H_0)$ for this inner product. We define an inner product on $E(H^-_0)$ by restriction of $(-,-)$.
	
	Let $\Sigma(G,H^-_0)$ be the set of restricted roots of $(G,H^-_0)$, i.e., the set of nontrivial characters of $H^-_0$ obtained by restriction of roots of $(G,H_0)$. We regard $\Sigma(G,H^-_0)$ as a subset of $E(H^-_0)$ by the inner product on $E(H^-_0)$. This is a root system of $E(H^-_0)$. Let $W_0$ be the Weyl group of $\Sigma(G,H^-_0)$. Set $\Psi_0=\Delta(G,H_0)\cap E(H^-_0)$. This is a root subsystem of $\Sigma(G,H^-_0)$. Let $W(\Psi_0)$ be the Weyl group of $\Psi_0$ which is a subgroup of $W_0$. Set $\cJ(\Psi_0)=\{w\in W(\Psi_0):~w^2=e\}$.
	
	Let $H$ be a $\theta$-stable maximal torus. We translate $H$ by $K$-conjugation to assume $H^-\subset H^-_0$. Regard $E(H^-)$ as a subspace of $E(H^-_0)$. Then one can find $c\in \cJ(\Psi_0)$ which is unique up to conjugation by $W_0$ such that
	$E(H^-)=\{v\in E(H^-_0):~cv=v\}$.
	In this way, we obtain a well-defined bijection
	\begin{equation}
		K(F)\backslash \Tor_G^\theta(F)\cong W_0\backslash\cJ(\Psi_0);~H\mapsto c.
		\label{eq:richardson--springer}
	\end{equation}

	\subsection{Proof of Theorem \ref{thm:tor}}\label{sec:tor}
	
	We denote the moduli scheme of pairs of Borel subgroups and their $\theta$-stable maximal tori by $\Kill_{G,\theta}$, i.e.,
	$\Kill_{G,\theta}=\Kill_G\times_{\Tor_G}\Tor_G^\theta$,
	where $\Kill_G\to \Tor_G$ is the projection. 
	
	\begin{lem}\label{lem:smooth}
		The $S$-schemes $\Tor_G^\theta$ and $\Kill_{G,\theta}$ are smooth.
	\end{lem}
	
	\begin{proof}
		This is immediate from \cite[Corollaire 5.8.3]{MR0218363} and Definition-Proposition \ref{defn-prop:x^theta}.
	\end{proof}
	
	Let $\cB^{\theta\mathrm{-}\std}_G$ be the moduli space of Borel subgroups $B$ which are at standard position with $\theta(B)$, i.e., a presheaf over the category of $S$-schemes defined by
	\begin{equation}
		T\mapsto \{B\in \cB_G(T):~(B,\theta(B))~\mathrm{is~of~standard~position}\}.
		\label{eq:defn}
	\end{equation}
	
	\begin{prop}\label{prop:smooth}
		\begin{enumerate}
			\item The $S$-space $\cB^{\theta\mathrm{-}\std}_G$ is represented by a smooth quasi-compact separated $S$-scheme.
			\item A Borel subgroup $B$ of $G$ lies in $\cB^{\theta\mathrm{-}\std}_G(S)$ if and only if $B$ \'etale locally admits a $\theta$-stable maximal torus.
		\end{enumerate}
	\end{prop}
	
	\begin{proof}
		Observe that $\Stand_{G,\emptyset,\emptyset}$ is \'etale locally expressed as the disjoint union of the $G$-orbits of the form $G/(B\cap B')$, where $(B,B')$ is a pair of Borel subgroups of standard position (recall \cite[Section 4.5.3]{MR0218364}). In particular, $\Stand_{G,\emptyset,\emptyset}$ is smooth, quasi-compact, and separated over $S$ (\cite[Proposition 4.5.1]{MR0218364}, Proposition \ref{prop:closed}, \cite[Corollaire 5.3.12]{MR0218363}, Lemma \ref{lem:quotient}).
		
		Define an involution $\eta$ on $\Stand_{G,\emptyset,\emptyset}$ by $(B_1,B_2)\mapsto (\theta(B_2),\theta(B_1))$. Then the assignment $B\mapsto (B,\theta(B))$ gives rise to an isomorphism
		$\cB^{\theta\mathrm{-}\std}_G\cong \Stand_{G,\emptyset,\emptyset}^\eta$.
		Part (1) now follows from Definition-Proposition \ref{defn-prop:x^theta}.
		
		We next prove (2). The ``if'' direction is evident by definition. Suppose that $(B,\theta(B))$ is of standard position. Then $B\cap \theta(B)$ is a smooth closed subgroup scheme of $G$. Moreover, if we write $\Tor_{B\cap \theta(B)}$ for the moduli space of maximal tori of $B\cap \theta(B)$, $\Tor_{B\cap \theta(B)}$ is represented by a smooth affine $S$-scheme (\cite[Proposition 4.5.1]{MR0218364}, \cite[Corollaire 5.4]{MR0232779}).
		
		Observe that $B\cap \theta(B)$ is $\theta$-stable in $G$. Hence $\Tor_{B\cap \theta(B)}$ is naturally equipped with an involution which we denote by $\theta$. Definition-Proposition \ref{defn-prop:x^theta} implies that $\Tor^\theta_{B\cap \theta(B)}$ is smooth over $S$. Moreover, its structure morphism is surjective by \cite[7.5. Theorem]{MR0230728}. Proposition \ref{prop:hensel} now implies that $B\cap \theta(B)$ and therefore $B$ \'etale locally admit a $\theta$-stable maximal torus. This proves the ``only if'' direction.	
	\end{proof}

	Let $s\in S$. Then we have a unique open $K_{\bar{s}}$-orbit $\cO_{\bar{s}}$ in $\cB_{G_{\bar{s}}}$.
	
	\begin{lem}\label{lem:effective}
		There exists a $K_s$-invariant open subvariety of $\cB_{G_s}$ which is a $\kappa(s)$-form of $\cO_{\bar{s}}$.
	\end{lem}

	\begin{proof}
		Pick any ample line bundle $\cL$ on $\cB_{G_s}$. Then we obtain an ample line bundle $\cL'$ on $\cO_{\bar{s}}$ by the base change of $\cL$ to $\kappa(\bar{s})$ and the restriction. Moreover, the canonical descent datum on $(\cB_{G_s} \otimes_{\kappa(s)}\kappa(\bar{s}),\cL\otimes_{\kappa(s)} \kappa(\bar{s}))$ restricts to that on $(\cO_{\bar{s}},\cL')$ since the open orbit is unique. This descent datum is effective by \cite[Section 6.1, Theorem 7]{MR1045822}.
	\end{proof}
	
	\begin{rem}
		More generally, let $G$ be a smooth affine algebraic group over a field $F$, and $X$ be a geometrically connected smooth algebraic $G$-variety. Let $\bar{F}$ be an algebraic closure of $F$. Assume that $X\otimes_F \bar{F}$ admits an open $G\otimes_F \bar{F}$-orbit $\cO'$ (e.g.~$X\otimes_F \bar{F}$ has finitely many $G\otimes_F \bar{F}$-orbits). Then a similar argument implies that there exists a unique $G$-invariant open subvariety $\cO\subset X$ such that $\cO\otimes_F \bar{F}=\cO'$. 
	\end{rem}

	Let $\cO_s$ be the open subvariety of Lemma \ref{lem:effective}.
	Since $\cO_s$ is nonempty and smooth, one can find a finite separable extension $F/\kappa(s)$ and
	\[B_F\in \cO_s(F)\subset\cB_G(F)=\cB^{\theta\mathrm{-}\std}_G(F)\]
	(see \cite[Lemme 4.1.1]{MR0218364} for the last equality). Replace $F$ with a certain finite separable extension of $F$ to choose a $\theta$-stable maximal torus $H_F$ of $B_F$ (Proposition \ref{prop:smooth} (2)). We lift it to $(B_0,H_0)\in\Kill_{G,\theta}(S')$ for a certain \'etale neighborhood $S'$ of $s\in S$ by application of Hensel's lemma of \cite[Proposition 1.10]{MR0215856}. Since the dimension of fibers is locally constant, we may replace $S'$ with an open neighborhood of $s'$ to assume that the $K\times_S S'$-orbit is open in $\cB_{G\times_S S'}$ (\cite[Corollaire (17.9.5)]{MR0238860}).
	
	Henceforth we assume that there exists $(B_0,H_0)\in\Kill_{G,\theta}(S)$ such that the attached $K$-orbit in $\cB_G$ is open. We may assume that $(G,H_0)$ admits a splitting, and that $B_0$ is attached to a positive system (\cite[Corollaire 5.5.5]{MR0218363}). We then localize $S$ in the Zariski topology to assume that $\theta$ acts on the set of roots. Using the splitting, we define $\cJ(\Psi_0)$ and $W_0$ in a similar way to the former section.
	
	We next construct a morphism $tt:\Tor_G^\theta\to (W_0\backslash\cJ(\Psi_0))_S$. For this, we may define a $W_0\backslash\cJ(\Psi_0)$-valued locally constant function on $S$ for each $\theta$-stable maximal torus $H$ of $G$ by passage to base changes and replacing $S$ accordingly. Firstly, assume that there exists an element $k\in K(S)$ such that $(kHk^{-1})^-\subset H^-_0$ and that $(kHk^{-1})^-$ admits a splitting compatible with that of $H_0$ (i.e., the containment is induced from a map between the corresponding lattices). We remark that this holds true locally in the \'etale topology of $S$ by Propositions \ref{prop:multiplicative}, \ref{prop:transp}, \cite[9.6. Lemma]{MR1066573}, \cite[Proposition B.3.4]{MR3362641}, and \cite[Corollaire 1.5]{MR0212024}. Then the element of $W_0\backslash\cJ(\Psi_0)$ corresponding to $H_{\bar{s}}$ is independent of choice of a geometric point $\bar{s}$ of $S$. We put the constant function on $S$ valued at this element. Thanks to the injectivity of the map \eqref{eq:richardson--springer}, the \'etale local construction of this constant function gives rise to a locally constant function on $S$, which is independent of the choices.

	\begin{lem}
		The morphism $tt$ is smooth surjective.
	\end{lem}
	
	\begin{proof}
		We see from Lemma \ref{lem:smooth} and \cite[Proposition 2.4.1 (iv)]{MR2791606} that $tt$ is smooth. To prove that $tt$ is surjective, we may pass to geometric fibers. Then the assertion follows from the bijection \eqref{defn-prop:x^theta}.
	\end{proof}
	
	Theorem \ref{thm:tor} (1) now follows from Propositions \ref{prop:type(R)}, \ref{prop:transp}, and the following general result:

	\begin{prop}\label{prop:singleorbit}
		Let $S$ be scheme, and $X$ be a smooth quasi-compact separated $S$-scheme, equipped with an action of a smooth affine group scheme $K$ over $S$. Let $f:X\to Y$ be a $K$-invariant surjective morphism of $S$-schemes with $Y$ finite \'etale. Assume the following conditions:
		\begin{enumerate}
			\item[(i)] For each geometric point $\bar{s}$ of $S$, every fiber of $\varphi_{\bar{s}}$ consists of a single $K_{\bar{s}}$-orbit.
			\item[(ii)] For every \'etale $S$-scheme $S'$ and an element $x\in X(S')$, the stabilizer subgroup of $K\times_S S'$ at $x$ is flat.
		\end{enumerate}
		Then $f$ descends to an isomorphism $K\backslash X\cong Y$. In particular, $K\backslash X$ is represented by a finite \'etale $S$-scheme.
	\end{prop}
	
	\begin{proof}
		Since the assertion is local in the \'etale topology of $S$, we may assume that $Y\cong V_S$ for a certain finite set $V$. Then we have
		$X\cong\coprod_{v\in V} f^{-1}(\{v\}_S)$.
		Therefore we may assume that $Y=S$.
		
		We may assume $X(S)\neq\emptyset$ by Proposition \ref{prop:hensel}. Pick $x\in X(S)$. 
		Then $Kx$ is represented by a smooth quasi-compact separated $S$-scheme (Proposition \ref{prop:closed}, Lemma \ref{lem:quotient}). It will suffice to prove that the orbit map $i:Kx\hookrightarrow X$ is an isomorphism. In view of \cite[Corollaire (17.9.5)]{MR0238860}, we may assume $S=\Spec F$ with $F$ an algebraically closed field. In this case, $i$ is an immersion (\cite[Proposition 7.17]{MR3729270}). The map $i$ is an isomorphism from (i) since a bijective immersion between reduced schemes is an isomorphism in general. This completes the proof.
	\end{proof}
	
	Part (2) is clear from the preceding argument.

	\begin{rem}
		One sees from the proof that the \'etale topology is enough for the quotient in Theorem \ref{thm:tor}.
	\end{rem}
	
	\begin{rem}\label{rem:localconstancy}
		In view of the proof of Proposition \ref{prop:singleorbit} or of Proposition \ref{prop:hensel}, one can and do pick an $S$-section $H_{[c]}$ of $tt^{-1}(\{[c]\}_S)$ for every conjugacy class
		\[[c]\in W_0\backslash\cJ(\Psi_0)\]
		after certain \'etale localization of $S$. It is evident by construction of $tt$ that
		\[K/N_{K}(H_{[c]})\cong tt^{-1}(\{[c]\}_S).\]
		
		Since the local constancy hypothesis is local in the \'etale topology of $S$, it is equivalent to assuming that $W_K(G,H_{[c]})$ is finite \'etale for every conjugacy class $[c]$ at each \'etale locus. In particular, this hypothesis holds true if $S=\Spec F$ for a field $F$ of characteristic not two.
	\end{rem}

	\section{Decomposition of $\cB^{\theta\mathrm{-}\std}_G$}
	
	This section is devoted to the proofs of Theorem \ref{mainthm} and Variant \ref{var:mainthm}. Let $S,G,\theta,K$ be as in Theorem \ref{thm:tor} unless specified otherwise (see Proposition \ref{prop:singleorbit}).
	
	\subsection{Proof of Theorem \ref{mainthm} (1) and (2)}\label{sec:thm1,2}
	Throughout this section, we assume the local constancy hypothesis. We plan to establish a $K$-orbit decomposition locally in the \'etale topology of $S$. The key is to verify that $\cB^{\theta\mathrm{-}\std}_G$ is ``closure relation free'':
	
	\begin{lem}\label{lem:transB}
		Let $B,B'\in \cB^{\theta\mathrm{-}\std}_G(S)$. Then
		$\STransp_K(B,B')$ is a smooth $S$-scheme. In particular, the following conditions are equivalent:
		\begin{enumerate}
			\renewcommand{\labelenumi}{(\alph{enumi})}
			\item There are an \'etale covering $\{S_\lambda\to S\}$ and $k_\lambda\in K(S_\lambda)$ such that
			\[k_\lambda (B\times_S S_\lambda)k^{-1}_\lambda=B'\times_S S_\lambda;\]
			\item There are an fpqc covering $\{S_\lambda\to S\}$ and $k_\lambda\in K(S_\lambda)$ such that
			\[k_\lambda (B\times_S S_\lambda)k^{-1}_\lambda=B'\times_S S_\lambda;\]
			\item For every geometric point $\bar{s}$ of $S$, there exists $k\in K(\kappa(\bar{s}))$ such that
			\[k B_{\bar{s}} k^{-1}= B'_{\bar{s}};\]
			\item The structure morphism $\STransp_K(B,B')\to S$ is surjective.
		\end{enumerate}
	\end{lem}
	
	For the proofs of this lemma and Theorem \ref{mainthm} (1), we can work locally in the \'etale topology of $S$. In view of Remark \ref{rem:localconstancy}, we may suppose that we are given finitely many $\theta$-stable maximal tori $H_{[c]}$ indexed by a finite set $\bar{\cJ}$ such that $[c]\mapsto H_{[c]}$ determines an isomorphism $\bar{\cJ}_S\cong K\backslash \Tor^\theta_G$. We localize $S$ to pick a splitting of $(G,H_{[c]})$ with $M_{[c]}$ the lattice of ``constant'' characters of $H_{[c]}$ (\cite[Corollaire 2.3]{MR0218363}). For later applications, let us fix a Borel subgroup $B_{[c]}$ attached to a positive system of $\Delta(G,H_{[c]})$. Work locally in the Zariski topology of $S$ to assume that $\theta$ respects $M_{[c]}$ and the set $\Delta(G,H_{[c]})$ of roots. Then identify $W(G,H_{[c]})$ with $W(\Delta(G,H_{[c]}))_S$ (\cite[Proposition 3.4]{MR0218363}). Thanks to the local constancy hypothesis, we may identify the canonical injective homomorphism $W_K(G,H_{[c]})\hookrightarrow W(G,H_{[c]})$ with the constant inclusion map $W_K(\Delta(G,H_{[c]}))_S\hookrightarrow W(\Delta(G,H_{[c]}))_S$ for a certain subgroup $W_K(\Delta(G,H_{[c]}))$. We fix a set $W(\Delta(G,H_{[c]}))'$ of complete representatives of the coset
	\[W_K(\Delta(G,H_{[c]}))\backslash W(\Delta(G,H_{[c]}))\]
	such that $W(\Delta(G,H_{[c]}))'\cap W_K(\Delta(G,H_{[c]}))=\{e\}$. Set $V\coloneqq \coprod_{[c]} W(\Delta(G,H_{[c]}))'$. For each $w\in W(\Delta(G,H_{[c]}))'$, we pick a representative in $N_G(H_{[c]})(S)$. Since $N_K(H_{[c]})$ is smooth over $S$, work locally in the \'etale topology of $S$ to assume that every element of $W_K(\Delta(G,H_{[c]}))$ admits a lift in $N_K(H_{[c]})(S)$ (Proposition \ref{prop:hensel}).
	
	\begin{proof}[Proof of Lemma \ref{lem:transB}]
		Since the statement is invariant under \'etale local $K$-conjugation of the Borel subgroups, we may assume that $B$ and $B'$ contains $H_{[c]}$ and $H_{[c']}$ for some $[c],[c']\in\bar{\cJ}$ respectively (Proposition \ref{prop:smooth} (2)). If $[c]\neq [c']$, $B$ and $B'$ are not $K$-conjugate to each other at any geometric point by \cite[4.4. Corollary]{MR803346}. In particular, we have $\STransp_K(B,B')=\emptyset$ in this case. This is a smooth $S$-scheme.
		
		Henceforth we let $[c]=[c']$. We may assume that $B'$ is a certain $W(\Delta(G,H_{[c]}))$-translation of $B$ by \cite[Corollaire 5.5.5]{MR0218363}. Translating $B'$ by $N_K(H_{[c]})(S)$, we may assume $B'=wBw^{-1}$ for some $w\in W(\Delta(G,H_{[c]}))'$. If $w\neq e$, $B$ and $B'$ are not $K$-conjugate to each other at any geometric point by Theorem \ref{thm:matsuki} (recall the choice on the complete representatives). In particular, we have $\STransp_K(B,B')=\emptyset$ in this case. This is a smooth $S$-scheme. If $w=e$, we obtain $B=B'$. In this case, we have $N_K(B)=B\cap K=(B\cap \theta(B))^\theta$, which is smooth over $S$ (\cite[Corollaire 5.8.3]{MR0218363} and Definition-Proposition \ref{defn-prop:x^theta}). This completes the proof.
		
	\end{proof}

	Under the current hypothesis, we have a morphism
	\[\coprod_{\overset{[c]\in\bar{\cJ}}{w\in W(\Delta(G,H_{[c]}))'}} K/(wB_{[c]}w^{-1}\cap K)\to \cB^{\theta\mathrm{-}\std}_G.\]
	This is an isomorphism from Theorem \ref{thm:matsuki} and Lemma \ref{lem:transB}. Take the quotient by $K$ to deduce $V_S\cong \rtype_\emptyset(G,\theta)$. This completes the proof of Theorem \ref{mainthm} (1). Part (2) is evident from the preceding argument.
	
	\begin{rem}
		We realize from the proof that the \'etale topology is enough for the representability of the quotient $\rtype_\emptyset(G,\theta)$ in (1).
	\end{rem}
	
	\subsection{Beilinson--Bernstein's affinity theorem}\label{sec:bb}
	
	Let $\varphi$ denote the composite map
	\[\cB^{\theta\mathrm{-}\std}_G\overset{(\id,\theta)}{\to}
	\Stand_{G,\emptyset,\emptyset}\overset{t_2}{\to}
	\StandType_{G,\emptyset,\emptyset},\]
	where $(\id,\theta)$ is defined by $B\mapsto (B,\theta(B))$. 
	
	\begin{lem}\label{lem:varphi_smooth}
		The morphism $\varphi$ is smooth.
	\end{lem}
	
	\begin{proof}
		This is immediate from the fact that $\StandType_{G,\emptyset,\emptyset}$ is (finite) \'etale and \cite[Proposition 2.4.1 (iv)]{MR2791606}.
	\end{proof}

	\begin{lem}\label{lem:affine}
		Let $w\in \StandType_{G,\emptyset,\emptyset}(S)$. Then $\varphi^{-1}(w)$ is affinely imbedded into $\cB_G$. 
	\end{lem}
	
	\begin{proof}
		Define a closed immersion $\cB_G\to \cB_G\times_S \cB_G$ by $B\mapsto (B,\theta(B))$ (recall that $\cB_G$ is separated over $S$). Let $\Delta_\theta$ be the corresponding closed subscheme. Then $\varphi^{-1}(w)$ is isomorphic to the intersection of $t^{-1}_{2}(w)$ and $\Delta_\theta$. Therefore it will suffice to show that the canonical map $t^{-1}_{2}(w)\hookrightarrow \cB_G\times_S \cB_G$ is an affine immersion.
		
		Since the assertion is local in the \'etale topology of $S$, we may assume that $G$ admits a splitting with $H$ a split maximal torus. Fix a positive system $\Delta^+(G,H)$. Then $\StandType_{G,\emptyset,\emptyset}$ is identified with $W(\Delta(G,H))_S$. One can then assume that $w$ is constant. 
		
		Let $B$ be the Borel subgroup attached to $\Delta^+(G,H)$. Let $\cB_{G,w}$ be the $B$-orbit attached to $w$ (cf.~\cite[Section 4.5.5]{MR0218364}). Identify the inclusion map $t^{-1}_2(w)\hookrightarrow\cB_G\times_S \cB_G$ with the map
		$G\times^B\cB_{G,w}\hookrightarrow G\times^B\cB_G$ which is induced from $\cB_{G,w}\subset\cB_G$. Observe that $\cB_{G,w}$ is affinely imbedded into $\cB_G$ by the arguments of \cite[Part II, Sections 13.1 and 13.2]{MR2015057}. It then follows from the faithfully flat descent that the map
		\[G\times^B\cB_{G,w}\hookrightarrow G\times^B\cB_G\]
		is an affine immersion (cf.~\cite[Part I, Section 5.14]{MR2015057}). This completes the proof.
	\end{proof}
	
	\begin{proof}[Proof of Theorem \ref{mainthm} (3)]
		Suppose that the fppf quotient $\rtype_\emptyset(G,\theta)$ is represented by a finite \'etale $S$-scheme. Without loss of generalities, we may assume that $\rtype_\emptyset(G,\theta)\cong V_S$ for a finite set $V$. Similarly, we may assume that it is isomorphic to the constant $S$-scheme $W_S$ for a certain finite set $W$ since $\StandType_{G,\emptyset,\emptyset}$ is finite \'etale. Since $\varphi$ is $K$-equivariant, it descends to a map
		\begin{equation}
			V_S\cong \rtype_\emptyset(G,\theta)\to\StandType_{G,\emptyset,\emptyset}\cong W_S.\label{eq:quot}
		\end{equation}
		Localize $S$ in the Zariski topology to assume that \eqref{eq:quot} is induced from a map
		\[\bar{\varphi}:V\to W.\]
		
		In view of Lemma \ref{lem:affine}, the proof will be completed by showing that the map
		$rt^{-1}(v)\hookrightarrow \varphi^{-1}(\bar{\varphi}_{S}(v))$
		is an affine immersion for every $v\in V_S(S)$. Actually, we may only prove it when $v$ is constant. For this, let $w\in W$. Then we have a canonical isomorphism
		$\coprod_{v\in \bar{\varphi}^{-1}(w)} rt^{-1}(\{v\}_S) \cong \varphi^{-1}(\{w\}_S)$
		by $\varphi=\bar{\varphi}_S\circ rt$. In particular, the map
		$rt^{-1}(\{v\}_S)\hookrightarrow \varphi^{-1}(\{w\}_S)$
		is an open and closed (thus affine) immersion into $\varphi^{-1}(\{w\}_S)$ for every $v\in \bar{\varphi}^{-1}(w)$. Run through all $w$ to deduce the assertion.
	\end{proof}

	\begin{rem}\label{rem:comparison}
		In the original proof by Beilinson--Bernstein, they considered a twisted action on $\cB_G\times_S\cB_G$. Here let us note the relation of the approaches of theirs and ours.
		
		We define an involution $\id\times \theta$ on $\cB_G\times_S \cB_G$ by $(B,B')\mapsto (B,\theta(B'))$. We write $\rho$ for the twist of the diagonal action of $G$ on $\cB_G\times_S \cB_G$ by this involution. Namely, $\rho$ is given by
		$\rho(g)(B,B')=(gBg^{-1},\theta(g)B'\theta(g)^{-1})$.
		Let $\Stand_{G,\theta,\emptyset,\emptyset}$ be the moduli $S$-space of pairs of Borel subgroups $(B,B')$ such that $(B,\theta(B'))$ are of standard position. Then $\Stand_{G,\theta,\emptyset,\emptyset}$ is identified with the pullback of $\Stand_{G,\emptyset,\emptyset}\subset \cB_G\times_S \cB_G$
		along the involution $\id\times \theta$. We put an action of $G$ on $\Stand_{G,\theta,\emptyset,\emptyset}$ by the transfer of the diagonal action on $\Stand_{G,\emptyset,\emptyset}$, which agrees with the restriction of $\rho$. In particular, the map $(B,B')\mapsto (B,\theta(B'))$ determines a $G$-equivariant isomorphism
		\begin{equation}
			\Stand_{G,\emptyset,\emptyset}\cong\Stand_{G,\theta,\emptyset,\emptyset}.\label{eq:isom1}
		\end{equation}
		We set $\StandType_{G,\theta,\emptyset,\emptyset}\coloneqq G\backslash\Stand_{G,\theta,\emptyset,\emptyset}$.
		It is evident that the isomorphism \eqref{eq:isom1} descends to
		\begin{equation}
			\StandType_{G,\emptyset,\emptyset}\cong
			\StandType_{G,\theta,\emptyset,\emptyset}.\label{eq:isom2}
		\end{equation}
		We denote the quotient map
		$\Stand_{G,\theta,\emptyset,\emptyset}\to \StandType_{G,\theta,\emptyset,\emptyset}$
		by $t_{2,\theta}$. Let \[(\id,\id):\cB^{\theta\mathrm{-}\std}_G\to\Stand_{G,\theta,\emptyset,\emptyset}\]
		be the diagonal map.
		Then we obtain a commutative diagram
		\[\begin{tikzcd}
			\cB^{\theta\mathrm{-}\std}_G\ar[r, "{(\id,\theta)}"]\ar[rd, "{(\id,\id)}"']
			\ar[rr, bend left, "\varphi"]
			&\Stand_{G,\emptyset,\emptyset}
			\ar[r, "t_2"]\ar[d, "\sim"{sloped,below}, "\eqref{eq:isom1}"]
			&\StandType_{G,\emptyset,\emptyset}\ar[d, "\sim"{sloped,below}, "\eqref{eq:isom2}"]\\
			&\Stand_{G,\theta,\emptyset,\emptyset}\ar[r, "t_{2,\theta}"]
			&\StandType_{G,\theta,\emptyset,\emptyset}.
		\end{tikzcd}\]
		Let $\varphi'$ be the composition of $\varphi$ and \eqref{eq:isom2}. Beilinson--Bernstein studied the intersection of $(\varphi')^{-1}(w)$ with the closed subscheme of diagonals in $\cB_G\times_S \cB_G$.
		
		To be more precise, they define $(\varphi')^{-1}(w)$ as a certain $G$-orbit and took the intersection in $\cB_G\times_S \cB_G$ (recall that they work over algebraically closed fields of characteristic zero). Since the orbit lies in $\Stand_{G,\theta,\emptyset,\emptyset}$, the resulting intersections of theirs and ours are same. For its scheme-theoretic decomposition, they saw the tangent spaces. On the other hand, we relate it with the scheme $\cB^{\theta\mathrm{-}\std}_G$ which we already knew that it is (\'etale locally) decomposed into $K$-orbits as an $S$-scheme.
		
		One could use $\varphi'$ for the proof of Theorem \ref{mainthm} (3). We adopted $\varphi$ because of consistency with \cite{MR1066573}; this will help us to use results of \cite{MR1066573} more directly in latter sections.
	\end{rem}

	\subsection{Image of the Springer morphism}\label{sec:var}
	
	In this section, we aim to prove Variant \ref{var:mainthm}. Let $\cI'_{G,\theta}$ be the image of $\varphi$. It is open in $\StandType_{G,\emptyset,\emptyset}$ since $\varphi$ is smooth (Lemma \ref{lem:varphi_smooth}). 
	
	We wish to prove Variant \ref{var:mainthm} (1) by using the combinatorial criterion for the image over algebraically closed fields of characteristic not two in \cite[7.13 Theorem]{MR1066573}. For this, we start with introducing the scheme of twisted involutions, based on \cite[1.8 Remark]{MR1066573}: Observe that the involution $\eta$ in the proof of Proposition \ref{prop:smooth} descends to that on $\StandType_{G,\emptyset,\emptyset}$, though it does not respect the diagonal action of $G$. We denote its fixed point closed subscheme by $\cI_{G,\theta}$. Then $\varphi$ factors through $\cI_{G,\theta}$. We remark that $\cI_{G,\theta}$ is \'etale locally expressed as follows: Suppose that we are given a $\theta$-stable maximal torus $H$ and a Borel subgroup $B$ containing $H$ (recall Lemma \ref{lem:smooth} for \'etale local existence). Then the map
	$N_G(H)\to \StandType_{G,\emptyset,\emptyset};~w\mapsto (B,wBw^{-1})$
	gives rise to an isomorphism
	\begin{equation}
		\iota_{B,H}:W(G,H)\cong \StandType_{G,\emptyset,\emptyset}.\label{eq:standtype}
	\end{equation}
	Let $w_\theta$ be the preimage of $(B,\theta(B))\in \StandType_{G,\emptyset,\emptyset}(S)$. 
	Under the identification \eqref{eq:standtype}, the involution on $W(G,H)$ obtained by transferring $\eta$ is given by $w\mapsto w^{-1}_\theta \theta(w)^{-1}w_\theta$.

	\begin{ex}\label{ex:I_G}
		Pick a splitting with a split maximal torus $H$, which is compatible with $\theta$. Identify $W(G,H)$ with $W(\Delta(G,H))_S$. Choose a positive system $\Delta^+(G,H)$ to define a Borel subgroup containing $H$. Then $w_\theta$ defined above is the constant section (i.e., an element of $W(\Delta(G,H))$) satisfying
		$\theta(\Delta^+(G,H))=w_\theta \Delta^+(G,H)$.
		In particular, the involution on $W(G,H)$ respects the constant sections. As a consequence, we get
		$\cI_{G,\theta}\cong\{w\in W(\Delta(G,H)):~\theta(w)w_\theta w=w_\theta\}_S$.
	\end{ex}

	\begin{proof}[Proof of Variant \ref{var:mainthm} (1)]
		Since the definition of $\cI'_{G,\theta}$ is compatible with base changes (use the fact that $\varphi$ is universally open if necessary), we may work locally in the \'etale topology of $S$ to suppose that we are given $(B,H)\in\Kill^\theta_G(S)$ (recall that $\Kill_G$ is smooth over $S$). Moreover, we may assume that $(G,H,B)$ admits a splitting which is compatible with $\theta$. We identify $\StandType_{G,\emptyset,\emptyset}$ with $W(\Delta(G,H))_S$. We also choose $B_0$ as in Section \ref{sec:tor}. Set $a_{\max}=\varphi(B_0)$. We may assume that $a_{\max}$ is constant.
		
		We now define the monoid $M$ and its twisted action $\ast$ on the set
		\[I_{G,\theta}\coloneqq\{w\in W(\Delta(G,H)):~\theta(w) w=e\}\]
		as in \cite[Section 3.10 and the paragraph below Section 3.11]{MR1066573}. Set
		\[I'_{G,\theta}\coloneqq\{a\in I_{G,\theta}:~a_{\max}\in M\ast a\}.\]	
		Then see geometric-fiberwisely to conclude that $\varphi$ maps onto $(I'_{G,\theta})_S$ by \cite[7.13 Theorem]{MR1066573}. This completes the proof.
	\end{proof}

	Variant \ref{var:mainthm} (2) then follows from Lemma \ref{lem:affine} and Proposition \ref{prop:singleorbit}.

	\section{Examples}\label{sec:ex}
	In this section, we see concrete examples. 
	
	\subsection{Example I}\label{sec:ex1}
	
	Let $S'\to S$ be a double Galois covering of schemes over $\bZ[1/2]$ with Galois involution $\bar{}$, $G_0$ be a reductive group scheme over $S$. Put
	\[G=\Res_{S'/S} (G_0\times_S S'),\]
	where $\Res_{S'/S}$ is the Weil restriction. Put an involution $\theta$ on $G$ by the Galois involution. Then the canonical map $G_0\to G$ gives rise to an isomorphism $G_0\cong K$.
	
	Observe that we have a splitting
	$G\times_S S'\cong (G_0\times_SS')^2$.
	Through this isomorphism, $\theta\times_S S'$ is identified with the switch of the factors. Therefore the assumption of Variant \ref{var:mainthm} holds true and 
	$\rtype_\emptyset(G,\theta)\cong\cI_{G,\theta}$ by \cite[10.1 Example]{MR1066573}.

	We wish to compute $\cI_{G,\theta}$. For simplicity, suppose that we are given a maximal torus $H_0\subset G_0$ and a Borel subgroup $B'_0\subset G_0\times_S S'$ containing $H_0\times_S S'$. Set 
	\[w_c=\iota^{-1}_{B'_0,H_0\times_S S'}(B'_0,\bar{B}'_0)\in 
	W(G_0\times_S S',H_0\times_S S')(S').\]
	There are two ways to compute $\cI_{G,\theta}$. One is to use the description of \cite[10.1 Example]{MR1066573} over $S'$ and to determine the induced Galois action. Identify $G\times_S S'$ with $(G_0\times_S S')^2$. Then $\iota_{(B'_0)^2,(H_0\times_S S')^2}$ induces an isomorphism of $\StandType^2_{G_0\times_S S',\emptyset,\emptyset}$ with $W(G_0,H_0)^2\times_S S'$. The induced involution is $(w,w')\mapsto ((w')^{-1},w^{-1})$. The attached fixed point closed subscheme of $W(G_0,H_0)^2\times_S S'$ is defined by $w'=w^{-1}$.
	It remains to determine the Galois action. For this, let $T$ be an $S$-scheme. Then the induced Galois action on $W(G_0,H_0)^2(T\times_S S')$ is given by
	\[(w,w')\mapsto (w^{-1}_c \bar{w}^{-1}w_c,w^{-1}_c (\bar{w}')^{-1}w_c),\]
	where $\bar{w}$ and $\bar{w}'$ are the usual conjugation.
	The first projection gives rise to an isomorphism from $\cI_{G,\theta}$ onto the closed subscheme of $\Res_{S'/S} (W(G_0,H_0)\times_S S')$ defined by $\bar{w}= w_cww^{-1}_c$.
	
	\begin{rem}\label{rem:Gal}
		The Galois involution on $W(G_0,H_0)\times_S S'$ induced from $\iota_{B'_0,H_0\times_S S'}$ and $\StandType_{G,\emptyset,\emptyset}$ is $w\mapsto w^{-1}_c\bar{w}w_c$, which is different from the involution appeared above. In particular, $\cI_{G,\theta}$ is not isomorphic to $W(G_0,H_0)$ in general, rather typically when $(G_0,H_0)$ is split.
	\end{rem}
	
	The second way is to realize $\cI_{G,\theta}$ directly as a subscheme of
	\[\Res_{S'/S} (W(G_0,H_0)\times_S S').\]
	Set $B=\Res_{S'/S} B'_0$ and $H=\Res_{S'/S} (H_0\times_S S')$ to define $(B,H)\in \Kill_{G_0}(S')$. Observe that we have a canonical isomorphism
	\begin{equation}
		W(G,H)\cong\Res_{S'/S} W(G_0\times_S S',H_0\times_S S')\label{eq:resweyl}
	\end{equation}
	(\cite[Lemma 4.1.26]{hayashijanuszewski}). Under the identifications of \eqref{eq:resweyl} and $\iota_{B,H}$ in \eqref{eq:standtype}, we have
	\[w_\theta=w_c\in (\Res_{S'/S} W(G_0\times_S S',H_0\times_S S'))(S).\]		
	The involution on $\Res_{S'/S} W(G_0\times_S S',H_0\times_S S')$ induced from $\theta$ is the Galois involution. With the description above Example \ref{ex:I_G}, we obtain a realization of $\cI_{G,\theta}$ as the fixed point subscheme of $\Res_{S'/S} W(G_0\times_S S',H_0\times_S S')$.

	\subsection{General observation on standard models}
	
	In the rest of this paper, we work with some of the standard $\bZ[1/2]$-forms $G$ of the classical Lie groups in \cite{MR4627704}. Henceforth we put $S=\Spec \bZ[1/2]$. We will also work over $\bZ[1/2,\sqrt{-1}]$ for \'etale local studies. To save space, put
	\[\begin{array}{ccc}
		k=\bZ[1/2],&k'=\bZ[1/2,\sqrt{-1}],&S'=\Spec k'.
	\end{array}\]
	Henceforth let $\theta$ be the standard $\bZ[1/2]$-form of a Cartan involution, i.e., the inverse of the transpose $(-)^T$ or the adjoint $(-)^\ast$ (see \cite[Section 1.5]{MR4627704}). Write $H_{\fun}$ for the standard fundamental Cartan subgroups given in \cite[Section 3.4]{MR4627704}. Let $w_0$ denote the longest element of the Weyl group or its lift given in \cite[Section 4]{MR4627704}.
	
	Let $G$ be any of the standard $\bZ[1/2]$-form of the classical Lie groups. We give a remark on the computation of $\cI_{G,\theta}$ with respect to $H_{\fun}\otimes_k k'$. Let
	\[\Delta^+(G\otimes_k k',H_{\fun}\otimes_k k')\]
	be the $\theta$-stable positive system constructed in \cite{MR4627704}. Then we have $w_\theta=e$ and
	\[\cI_{G,\theta}\otimes_k k'\cong
	\{w\in W(\Delta(G\otimes_k k',H_{\fun}\otimes_k k')):~\theta(w) w=e\}_{S'}\]
	(\cite[\S 3]{MR803346}, \cite[Section 1.6]{MR1066573}). To compute the $k$-form, observe that the transferred Galois involution on $W(\Delta(G\otimes_k k',H\otimes_k k'))$ from the identification
	\[\StandType_{G,\emptyset,\emptyset}\otimes_k k'
	\cong W(\Delta(G\otimes_k k',H_{\fun}\otimes_k k'))_{S'}\]
	is $w\mapsto w_0\bar{w}w_0$ (Remark \ref{rem:Gal}).
	
	\begin{lem}\label{lem:conj}
		We have $\bar{w}=\theta(w)$ for any $w\in W(\Delta(G\otimes_k k',H_{\fun}\otimes_k k'))$. In particular, we have $\bar{w}=w$ if $\theta$ is inner (over the field $\bC$ of complex numbers).
	\end{lem}
	
	\begin{proof}
		This is straightforward:
		\[\begin{split}
			\bar{w}\Delta^+(G\otimes_k k',H_{\fun}\otimes_k k')
			&=\overline{w\overline{\Delta^+(G\otimes_k k',H_{\fun}\otimes_k k')}}\\
			&=-\overline{w\Delta^+(G\otimes_k k',H_{\fun}\otimes_k k')}\\
			&=\theta(w\Delta^+(G\otimes_k k',H_{\fun}\otimes_k k'))\\
			&=\theta(w)\Delta^+(G\otimes_k k',H_{\fun}\otimes_k k')
		\end{split}\]
		(take the base change to $\bC$ to see the third equality).
	\end{proof}
	
	\begin{cor}\label{cor:Galois}
		The Galois involution on the set
		\[\{w\in W(\Delta(G\otimes_k k',H_{\fun}\otimes_k k')):~\theta(w) w=e\}\]
		is given by $w\mapsto w_0w^{-1}w_0$.
	\end{cor}

	Let us also give a note on computation of the Galois action on $\rtype_\emptyset(G,\theta)\otimes_k k'$: Let $H$ be a $\theta$-stable maximal torus of $G$ such that $(G\otimes_k k',H\otimes_k k')$ is split. Choose a positive system $\Delta^+(G\otimes_k k',H\otimes_k k')$. Let $B'$ be the attached Borel subgroup. Take the Weyl group element $u$ with
	\[\overline{\Delta^+(G\otimes_k k',H\otimes_k k')}
	=u\Delta^+(G\otimes_k k',H\otimes_k k').\]
	Pick its lift in $N_G(H)(k')$. For $w\in N_G(H)(k')$, we have
	$\overline{w B'w^{-1}}=\bar{w} uB'u^{-1}\bar{w}^{-1}$.
	In this way, $W(G\otimes_k k',H\otimes_k k')$ is endowed with a Galois involution for $w\mapsto \bar{w} u$.
	
	\begin{ex}\label{ex:conj_fun}
		Put $H=H_{\fun}$. In this case, $u=w_0$ and $\bar{w}=\theta(w)$ for each element
		$w\in W(G,H_{\fun})(k')$.
		Assume that $W_K(G,H_{\fun})\otimes_k k'$ is constant, and that $\theta$ is trivial on $H_{\fun}$. For an element $w\in W(G,H_{\fun})(k')$, we have $\bar{w}u\in W_K(G,H)(k')w $ if and only if $ww_0 w^{-1}\in W_K(G,H)(k')$.
	\end{ex}
	
	For more practical computations, suppose that we are given another $\theta$-stable maximal torus $H_i$ and an element $g_i\in G(k')$ such that
	$g_i (H\otimes_k k')g^{-1}_i=H_i\otimes_k k'$ (here read $i$ just as a formal symbol). In particular, $(G\otimes_k k',H_i\otimes_k k')$ is split. We define a positive system $\Delta^+(G\otimes_k k',H_i\otimes_k k')$ of $(G\otimes_k k',H_i\otimes_k k')$ by transferring that of $(G\otimes_k k',H\otimes_k k')$.
	Namely, $\Delta^+(G\otimes_k k',H_i\otimes_k k')$ consists of the characters of $H_i\otimes_k k'$ defined by
	$g_i\alpha:h\mapsto \alpha(g^{-1}_i hg_i)$ ($\alpha\in\Delta^+(G\otimes_k k',H\otimes_k k')$).
	The transferred Galois involution described above back to $W(G\otimes_k k',H\otimes_k k')$ is given as follows: Set $u'_i=g^{-1}_i \bar{g}_i$. Then $u'_i$ belongs to $N_G(H)(k')$. We denote the corresponding element of the Weyl group by the same symbol.
	
	\begin{prop}\label{prop:transfer}
		In the setting above, the transferred Galois involution on the Weyl group scheme
		$W(G\otimes_k k',H\otimes_k k')$
		is given by $w\mapsto u'_i\bar{w}u$. 
	\end{prop}
	
	\begin{proof}
		Firstly, observe that the Galois involution on $W(G\otimes_k k',H_i\otimes_k k')$ is given by $w\mapsto \bar{w}\bar{g}_iug^{-1}_i$. In fact, let $B'_i$ be the Borel subgroup attached to $\Delta^+(G\otimes_k k',H_i\otimes_k k')$. Then $B'_i=g_i B' g^{-1}_i$ and we have
		\[\overline{B}'_i
		=\bar{g}_i \bar{B} \bar{g}^{-1}_i
		=\bar{g}_i u B u^{-1} \bar{g}_i
		=g_iu'_iu g^{-1}_iB'_i g_iu^{-1} (u'_i)^{-1} g^{-1}_i
		=\bar{g}_iug^{-1}_iB'_ig_i u^{-1}\bar{g}^{-1}_i.
		\]
		
		Let $w\in N_G(H)(k')$. The transfer of $w$ to $N_G(H_i)(k')$ is $g_i w g^{-1}_i$. Apply the Galois involution given above to obtain
		$\bar{g}_i \bar{w} ug^{-1}_i$.
		Transfer it back to $N_G(H)(k')$ to deduce the assertion.
	\end{proof}
	
	\subsection{Example II}

	Put $G=\GL_n$ ($n\geq 1$). In this case, the assumption of Variant \ref{var:mainthm} holds and $\rtype_\emptyset(G,\theta)\cong \cI_{G,\theta}$ by \cite[10.2 Example]{MR1066573}. Let $H_0$ be the split maximal torus of diagonal matrices. Together with Variant \ref{var:mainthm}, the fact that $(G,H_0)$ is split over $\bZ[1/2]$ implies that the complex $K$-orbits on the complex flag variety of $\GL_n$ are defined over $\bZ[1/2]$. Based on Example \ref{ex:I_G}, we can describe $\cI_{G,\theta}$ as follows: Notice that $\theta$ acts trivially on the Weyl group $W(\Delta(G,H_0))$. For a positive system, we choose the standard one. Then $w_\theta$ is the longest element $w_0$. For computations, we identify $W(\Delta(G,H_0))$ with $\fS_n$. Then we have an isomorphism
	$\cI_{G,\theta}\cong\{w\in \fS_n:~(ww_0)^2=e\}_S$.
	
	\begin{ex}[{\cite[Theorem 1.1]{hayashikgb}}]
		Put $n=3$. Then $w_0=(1~3)$ and
		\[\{w\in \fS_n:~(ww_0)^2=e\}=\{e,(1~2~3),(1~3~2),(1~3)\}.\]
		We totally have four homogeneous $K$-subschemes in $\cB^{\theta\mathrm{-}\std}_G$. The open and closed subschemes correspond to $(1~3)$ and $e$ respectively.
	\end{ex}

	\subsection{Example III}
	Put $G=\Uu^\ast(2n)$. Identify
	$W(\Delta(G\otimes_k k',H_{\fun}\otimes_k k'))$
	with $\fS_{2n}$ through the isomorphism
	$\Uu^\ast(2n)\otimes_k k'\cong \GL_{2n}$
	in \cite[Example 3.3.4]{MR4627704}. 
	In view of \cite[10.4 Example]{MR1066573}, $G\otimes_k k'$ satisfies the condition of Variant \ref{var:mainthm} and
	\[\rtype_\emptyset(G,\theta)\otimes_k k'\cong (\cJ_0w_0)_{S'},\]
	where $\cJ_0$ is the set of fixed point free involutions in $\fS_{2n}$. Therefore the Galois involution is trivial. This shows
	$\rtype_\emptyset(G,\theta)\cong (\cJ_0w_0)_{S}$.

	\subsection{Example IV}
	Put $G=\SL_{2n}$ with $n\geq 1$. Define a torus $H^{\mathrm{bl}}\subset\GL_2$ by
	\[H^{\mathrm{bl}}(R)=\left\{\left(\begin{array}{cc}
		a&b\\
		-b&a
	\end{array}\right)\in\GL_2(R)\right\},\]
	where $R$ runs through commutative $k$-algebras. Set 
	\[g^{\bl}\coloneqq\left(\begin{array}{cc}
		\frac{1}{2} & \sqrt{-1} \\
		\frac{\sqrt{-1}}{2} & 1
	\end{array}\right)\in\SL_2(k')\]
	Then we have
	\[H^{\mathrm{bl}}\otimes_k k'\cong g^{\bl}\diag(\GL_1,\GL_1)(g^{\bl})^{-1}.\]
	
	For $0\leq i\leq n$, define a $\theta$-stable maximal torus $H_i\subset\SL_{2n}$ by
	\[H_i=\diag(\overbrace{H^{\mathrm{bl}},H^{\mathrm{bl}},\ldots,H^{\mathrm{bl}}}^{i},\overbrace{\GL_1,\GL_1,\ldots,\GL_1}^{2n-2i})\cap \SL_{2n}.\]
	We notice that $H_0$ is the split maximal torus of the diagonal matrices. For $0\leq i\leq n$, $H_i\otimes_k k'$ is the conjugate of $H_0\otimes_k k'$ by 
	\begin{equation}
		g(i)\coloneqq\diag(\overbrace{g^{\bl},g^{\bl},\ldots,g^{\bl}}^{i},\overbrace{1,1,\ldots,1}^{2n-2i}).
		\label{eq:translation_matrix}
	\end{equation}
	Moreover, they give rise to an isomorphism $\Tor^\theta_{\SL_{2n}}\cong \{H_0,H_1,\ldots,H_n\}_S$.
	
	To check the local constancy hypothesis, we translate $W_K(G,H_i)\otimes_k k'$ by the matrix \eqref{eq:translation_matrix} to identify it with a subgroup of $W(G,H_0)\otimes_k k'\cong (\fS_{2n})_{S'}$ for $0\leq i\leq n$. If $0\leq i\leq n-1$, let $W_i\subset\fS_{2n}$ be the subgroup generated by $(2j-1~2j)$ ($1\leq j\leq i$), $(2j-1~2j+1)(2j~2j+2)$ ($1\leq j\leq i-1$), and $(j~j+1)$ ($2i+1\leq j\leq 2n-1$). We also let $W_n\subset\fS_{2n}$ be the subgroup generated by $(2j-1~2j)(2j+1~2j+2)$ ($1\leq j\leq n-1$) and $(2j-1~2j+1)(2j~2j+2)$ ($1\leq j\leq n-1$). Then $W_K(G,H_i)\otimes_k k'$ is naturally identified with the constant subgroup $(W_i)_{S'}$. In particular, the local constancy hypothesis holds.
	
	It remains to compute the Galois involution. The direct computation shows $g^{-1}(i)\overline{g(i)}$ represents $u'_i\coloneqq (1~2)(3~4)\cdots(2i-1~2i)\in\fS_{2n}$.
	In view of Proposition \ref{prop:transfer}, the induced Galois involution on $W_i\backslash \fS_{2n}$ is $W_iw\mapsto W_{i}u_i'w$. If $i\neq n$ or $n$ is even then $u'_i\in W_i$ and the Galois involution is trivial; Otherwise, i.e., if $n$ is odd and $i=n$ then $u'_i\not\in W_i$ and the Galois involution is free. This shows
	\[K\backslash\cB_{\SL_{2n}}\cong\begin{cases}
		\coprod_{0\leq i\leq n} (W_{i}\backslash\fS_{2n})_S&(n~\mathrm{is~even})\\
		(\Gamma\backslash (W_{n}\backslash\fS_{2n}) )_{S'}\coprod
		\coprod_{0\leq i\leq n-1} (W_{i}\backslash\fS_{2n})_S&(n~\mathrm{is~odd}),
	\end{cases}\]
	where $\Gamma\backslash (W_{n}\backslash\fS_{2n})$ is the set of Galois orbits in $W_{n}\backslash\fS_{2n}$.

	\subsection{Example V}
	Put $G=\SO(2n+1,1)$. It is easy to check
	\[\begin{array}{c}
		K\backslash\Tor^\theta_G\cong \{H_{\fun}\}_S= S,\\
		W_K(G,H_{\fun})= N_{K}(H_{\fun})/(H_{\fun}\cap K)\cong 
		N_{K^0}(H_{\fun})/(H_{\fun}\cap K^0)\cong W(K^0,T),
	\end{array}\]
	where $K^0$ is the unit component of $K$ (\cite[D\'efinition 3.1]{MR0234961}), and $T$ is the maximal torus of $K^0$ given in \cite{MR4627704}. In particular, the local constancy hypothesis holds. Set 
	\[W\coloneqq\{((\epsilon_i),\sigma)\in\{\pm 1\}^{n+1}\rtimes\fS_{n+1} :~
	\prod_{1\leq i\leq n} \epsilon_i=1\},\]
	\[W_K\coloneqq\{((\epsilon_i),\sigma)\in\{\pm 1\}^{n+1}\rtimes\fS_{n}:~
	\prod_{1\leq i\leq n} \epsilon_i=1\}\subset W.\]
	Here $\fS_n$ is regarded as a subgroup of $\fS_{n+1}$ for $\sigma(n+1)=n+1$ ($\sigma\in\fS_n$). Then $W_K(G,H_{\fun})\otimes_k k'\hookrightarrow W(G,H_{\fun})\otimes_k k'$ is identified with $(W_K)_S\subset W_S$. It is elementary that
	$W'\coloneqq\{(i~n+1)\in\fS_{n+1}:~1\leq i\leq n+1\}\subset W$
	is a complete system of representatives of $W_K\backslash W$. In particular, every element of $W_K\backslash W$ is represented by a transposition.
	
	It remains to compute the Galois action on
	$\rtype_\emptyset(G,\theta)\otimes_k k'\cong W'_{S'}$.
	Notice that under the identification of the character group of $H\otimes_k k'$ with $\bZ^{n+1}$, $\theta$ acts on $\bZ^{n+1}$ by $\diag(I_n,-1)$.
	\begin{description}
		\item[Case 1: $n$ is even] $w_0=\diag(-I_n,1)$ and
		\[\begin{split}
			\bar{w}w_0w^{-1}&=\diag(I_n,-1)w\diag(I_n,-1)w_0w\\
			&=\diag(I_n,-1)w(-I_{n+1})w\\
			&=\diag(-I_n,1)\in W_K
		\end{split}\]
		for $w=(i~n+1)\in W'$ with $1\leq i\leq n+1$.
		\item[Case 2: $n$ is odd] $w_0=-1$ and
		\[\begin{split}
			\bar{w}w_0w^{-1}&=\diag(I_n,-1)w\diag(I_n,-1)w_0w^{-1}\\
			&=\diag(I_n,-1)w\diag(-I_n,1)w^{-1}\\
			&=\diag(I_n,-1)\diag(-I_{i-1},1,-I_{n+1-i})\\
			&=\begin{cases}
				\diag(-I_{i-1},1,-I_{n-i},1)&(1\leq i\leq n)\\
				-1&(i=n+1)
			\end{cases}\\
			&\in W_K
		\end{split}\]
		for $w=(i~n+1)\in W'$ with $1\leq i\leq n+1$.
	\end{description}
	Hence the Galois involution is trivial in both cases and we get
	$\rtype_\emptyset(G,\theta)\cong W'_S$.
	
	\subsection{Example VI}
	Put $G=\SO(2n,1)$ with $n\geq 1$. In this case, $\theta$ acts trivially on $H_{\fun}$ and $W(G,H_{\fun})=W_K(G,H_{\fun})$ (see geometric-fiberwisely). In particular, the assumption of Variant \ref{var:mainthm} holds by \cite[9.15 Corollary]{MR1066573}.
	
	To compute the image of $\varphi$, set
	\[H_0=\diag(\overbrace{\SO(2),\SO(2),\ldots,\SO(2)}^{n-1},1,\SO(1,1))\subset G,\]
	\[g=\diag\left(I_{2n-1},\left(\begin{array}{ccc}
		0 & 0 & -\sqrt{-1} \\
		1 & 0 & 0 \\
		0 & \sqrt{-1} & 0
	\end{array}\right)\right)\in G(k').\]
	Then $H_0$ is $\theta$-stable and
	$K\backslash\Tor^\theta_G\cong\{H_{\fun},H_0\}_S$.
	We also notice that
	\[g(H_{\fun}\otimes_k k')g^{-1}=H_0\otimes_k k'.\]
	Let $B'$ be the $\theta$-stable Borel subgroup attached to the $\theta$-stable positive system of \cite{MR4627704}. Then $gwB'w^{-1}g^{-1}$ with $w\in N_G(H_{\fun})(k')$ geometric-fiberwisely list all the Borel subgroups containing $H_0\otimes_k k'$. For each $w$, the image under $\varphi$ is computed as
	\begin{flalign*}
		&\varphi(gwB'w^{-1}g^{-1})\\
		&=(gwB'w^{-1}g^{-1},\theta(gwB'w^{-1}g^{-1}))
		(B',w^{-1}g^{-1}\theta(g)wB'w^{-1}\theta(g)^{-1}gw).
	\end{flalign*}
	Write $W=\{\pm 1\}^n\rtimes\fS_{n}$ to identify $\StandType_{G,\emptyset,\emptyset}\otimes_kk'$ with $W_{S'}$. Then we have
	\[g^{-1}\theta(g)=\diag(I_{2n-1},-I_2)\in N_G(H_{\fun})(k'),\]
	which corresponds to $(1,1,\ldots,1,-1)\in W$. Hence if we set
	\[I'=\{(\overbrace{1,1,\ldots,1}^{i},-1,\overbrace{1,1,\ldots,1}^{n-i-1})\in W:~
	0\leq i\leq n\}\]
	(if $i=n$, the corresponding element is the unit of $W$),
	$\varphi\otimes_k k'$ maps onto $I'_{S'}$ under the identification
	$\StandType_{G,\emptyset,\emptyset}\otimes_k k'\cong W_{S'}$.
	Since $w_0=-1$, the Galois involution is trivial. We thus conclude
	$\rtype_\emptyset(G,\theta)\cong I'_S$.
	
	\subsection{Example VII}
	
	Put $G=\Uu(p,q)$ with $p\geq q$. Write $n=p+q$. Firstly, recall that we have a natural isomorphism
	\begin{equation}
		\Uu(p,q)\otimes_k k'\cong\GL_n,\label{eq:complexification}
	\end{equation}
	whose restriction to $H_{\fun}\otimes_k k'$ is onto the maximal torus of diagonal matrices in $\GL_n$ (\cite[Examples 3.3.2, 3.4.3]{MR4627704}).	
	
	We next construct non-fundamental $\theta$-stable maximal tori of $G$. For this, recall that the conjugate action on $k'$ over $k$ induces an involution on $\Res_{k'/k}\GL_1$ which we denote by $\bar{}$. Let $H^{\bl}\subset\Uu(1,1)$ be the image of the monomorphism \[\Res_{k'/k} \GL_1\hookrightarrow \Uu(1,1)\]
	define by
	\[z\mapsto \left(\begin{array}{cc}
		1 & -1 \\
		1 & 1
	\end{array}\right)\diag(\bar{z},z^{-1})\left(\begin{array}{cc}
		1 & -1 \\
		1 & 1
	\end{array}\right)^{-1}
	=\frac{1}{2}\left(\begin{array}{cc}
		\bar{z}+z^{-1} & \bar{z}-z^{-1} \\
		\bar{z}-z^{-1} & \bar{z}+z^{-1}
	\end{array}\right).
	\]
	For $0\leq i\leq q$, let $H_i$ be the $\theta$-stable maximal torus consisting of the matrices $(a_{st})$ with the following properties:
	\begin{enumerate}
		\item[(i)] For $1\leq j\leq p-q+i$ or $p+1\leq j\leq p+i$, we have $a_{jj}\in\Uu(1)$,
		\item[(ii)] For $1\leq j\leq i$, 
		\[\left(\begin{array}{cc}
			a_{p-q+i+j~p-q+i+j} & a_{p-q+i+j~n-q+i+j} \\
			a_{n-q+i+j~p-q+i+j} & a_{n-q+i+j~n-q+i+j}
		\end{array}\right)\in H^{\bl},\]
		\item[(iii)] the other entries are zero.
	\end{enumerate}
	Then we have $K\backslash\Tor^\theta_G\cong \{H_0,\ldots,H_q\}_S$.
	
	To translate $H_q=H_{\fun}$ to $H_i$ over $k'$, set $g_i\in G(k')$ as
	\[g_i=\left(\begin{array}{cccc}
		I_{p-q+i} & 0 & 0 & 0 \\
		0 & \frac{1}{4}(3\otimes 1-\sqrt{-1}\otimes \sqrt{-1})I_{q-i} & 0 & -\frac{1}{4}(-1\otimes 1+3\sqrt{-1}\otimes \sqrt{-1})I_{q-i} \\
		0 & 0 & I_{i} & 0 \\
		0 & \frac{1}{4}(1\otimes 1-3\sqrt{-1}\otimes \sqrt{-1})I_{q-i} & 0 & \frac{1}{4}(3\otimes 1-\sqrt{-1}\otimes \sqrt{-1})I_{q-i}
	\end{array}\right).\]
	Then we have $H_i\otimes_k k'\cong g_i (H_q\otimes_k k')g^{-1}_i$. To check this, it should deserve to note the following basic observation: Under the isomorphism \eqref{eq:complexification}, $H^{\bl}\otimes_k k'$ is identified with the torus
	\[R\mapsto \left\{\left(\begin{array}{cc}
		x&y\\
		y&x
	\end{array}\right)\in\GL_2(R)\right\},\]
	where $R$ runs through all commutative $k'$-algebras. Let $g^{\mathrm{bl}}\in\Uu(1,1;k')$ be the preimage of 
	\[\left(\begin{array}{cc}
		1& -1\\
		1& 1
	\end{array}\right)\in\GL_2(k)\]
	along the isomorphism \eqref{eq:complexification}. Explicitly, we have
	\[g^{\mathrm{bl}}=\frac{1}{4}\left(\begin{array}{cc}
		3\otimes 1-\sqrt{-1}\otimes \sqrt{-1} & -1\otimes 1+3\sqrt{-1}\otimes \sqrt{-1} \\
		1\otimes 1-3\sqrt{-1}\otimes \sqrt{-1} & 3\otimes 1-\sqrt{-1}\otimes \sqrt{-1}
	\end{array}\right).\]
	One can easily check
	$g^{\mathrm{bl}}\diag(\Uu(1)\otimes_k k',\Uu(1)\otimes_k k')(g^{\mathrm{bl}})^{-1}
	=H^{\mathrm{bl}}\otimes_k k'$
	(see it in $\GL_2$).
	
	Let $W_{K,i}$ be the subgroup of $\fS_{n}$ generated by $\fS_{p-q+i}$, $\fS_{q-i}\ltimes\{\pm 1\}^{q-i}$, and $\fS_{i}$, which are regarded as subgroups of $\fS_n$ as follows:
	\begin{itemize}
		\item For $\sigma\in\fS_{p-q+i}$,
		\[\sigma(j)=\begin{cases}
			\sigma(j)&(1\leq j\leq p-q+i)\\
			j&(p-q+i+1\leq j\leq n).
		\end{cases}\]
		\item For $\sigma\in\fS_{q-i}$,
		\[\sigma(j)=\begin{cases}
			\sigma(j-(p-q+i))+(p-q+i)&(p-q+i+1\leq j\leq p)\\
			\sigma(j-(n-q+i))+(n-q+i)&(n-q+i+1\leq j\leq n)\\
			j&(\mathrm{otherwise}).
		\end{cases}\]
		\item For $(\epsilon_s)\in\{\pm 1\}^i$,
		\[(\epsilon_s)(j)=\begin{cases}
			j+q&(p-q+i+1\leq j\leq p,~\epsilon_{j-(p-q+i)}=-1)\\
			j-q&(n-q+i+1\leq j\leq n,~\epsilon_{j-(n-q+i)}=-1)\\
			j&(\mathrm{otherwise}).
		\end{cases}\]
		\item For $\sigma\in \fS_{i}$,
		\[\sigma(j)=
		\begin{cases}
			\sigma(j-p)+p&(p+1\leq j\leq n-q+i)\\
			j&(\mathrm{otherwise}).
		\end{cases}
		\]
	\end{itemize}
	Then one can check $g^{-1}_i(W_K(G,H_i)\otimes_k k')g_i\cong (W_{K,i})_{S'}$.
	In particular, the local constancy hypothesis holds. We also have
	\[\rtype_\emptyset(G,\theta)\otimes_k k'\cong
	\coprod_{0\leq i\leq q} (W_{K,i}\backslash\fS_n)_{S'}.\]
	Observe that if we identify the Weyl group of $\Uu(1,1)\otimes_k k'$ with respect to the maximal torus of diagonal matrices with $\fS_2$,
	$(g^{\bl})^{-1}\bar{g}^{\bl}$
	represents $(1~2)\in\fS_2$. Therefore $u'_i$ in Proposition \ref{prop:transfer} is
	\[(p-q+i+1~n-q+i+1)(p-q+i+2~n-q+i+2)\cdots(p~n)\in W_{K,i}.\]
	As a result, the Galois involution is given by
	$W_{K,i}w\mapsto W_{K,i}ww_0$
	(recall that $\bar{w}=w$ since the conjugate action on the character group of $H_{\fun}\otimes_k k'$ is $-1$).
	
	The coset $W_{K,i}w$ is fixed by the Galois involution if and only if $ww_0w^{-1}\in W_{K,i}$. Set
	\[I_i=\{W_{K,i}w\in W_{K,i}\backslash \fS_n:~ww_0w^{-1}\in W_{K,i}\},\]
	\[I'_i=\{W_{K,i}w\in W_{K,i}\backslash \fS_n:~ww_0w^{-1}\not\in W_{K,i}\}.\]
	Let $\Gamma\backslash I'$ be the set of Galois orbits in $I'$. Then we have
	\[\rtype_\emptyset(G,\theta)
	\cong \coprod_{0\leq i\leq q}
	((I_i)_S\coprod (\Gamma\backslash I'_i)_{S'}).\]
	
	\begin{rem}
		Notice that the conjugacy class of $w_0$ consists of the fixed point free involutions (resp.~the involutions, each of which has a unique fixed element) if $n$ is even (resp.~odd). It is easy to show that $I_i=\emptyset$ if and only if both $p-q$ and $i$ are even.
	\end{rem}
	
	Let us record:
	
	\begin{prop}\label{prop:S_0}
		Let $S_0\subset\fS_n$ be the set of involutions with $(p-q)$ fixed points. 
		\begin{enumerate}
			\item The element $g^{-1}_0\theta(g_0)\in N_G(H_{\fun})(k')$ represents $\prod_{j=1}^{q}(p-q+j~n-q+j)$ (recall $p<n-q+1$ to see that the product is independent of the order).
			\item The map
			\[W_{K,0}\backslash \fS_n\to \fS_n;~w\mapsto w^{-1} \left(\prod_{i=j}^{q}(p-q+j~n-q+j)\right)w\]
			is a bijection onto $S_0$.
			\item The open subscheme $(\varphi\otimes_kk')^{-1}((S_0)_{S'})\subset\cB^{\theta\mathrm{-}\std}_{G\otimes_k k'}$ consists of the Borel subgroups which contain $H_0\otimes_k k'$ up to \'etale local $(K\otimes_k k')$-conjugacy. Moreover, it descends to an isomorphism
			\[(W_{K}(G,H_0)\backslash W(G,H_0))\otimes_k k'\cong 
			(K\otimes_k k')\backslash(\varphi\otimes_kk')^{-1}((S_0)_{S'})
			\cong (S_0)_{S'}.
			\]
			\item The Galois action on the $S'$-scheme $(\fS_n)_{S'}\cong\StandType_{G,\emptyset,\emptyset}\otimes_k k'$ is given by
			$w\mapsto w_0ww_0$.
		\end{enumerate}
	\end{prop}
	
	We recall that $(S_0)_{S'}\subset(\fS_n)_{S'}$. Since $\varphi$ is defined over $k$, the Galois action on $(S_0)_{S'}$ induced from the isomorphism of (3) is computed by (4). The advantage of this expression over the involution on $W_{K,0}\backslash \fS_n$ is that no quotient is involved.
	
	\begin{proof}
		The assertions (1) and (2) are elementary. For (1), use the fact that $\theta$ is equal to the conjugation by $\diag(I_p,-I_q)$, and compute $g^{-1}_0\theta(g_0)$ in $\GL_n(k')$ if necessary. Part (3) follows as a consequence of (1), (2), and Proposition \ref{prop:singleorbit} (recall the definition of $\varphi$). For (4), see Remark \ref{rem:Gal}. Since the conjugation acts on the character group of $H_{\fun}\otimes_kk'$ is $-1$. We thus get $\bar{w}=w$ for $w\in\fS_n$.
	\end{proof}

	\begin{ex}
		Put $q=1$. Then we have
		$S_0=\{(i~j)\in\fS_{p+1}:~1\leq i<j\leq p+1\}$.
		The involution on $S_0$ determined by Proposition \ref{prop:S_0} (4) is
		\[(i~j)\mapsto (p+2-i~p+2-j)~(1\leq i<j\leq p+1).\]
		Since $i<j$, the transpose $(i~j)$ is fixed by this involution if and only if $i+j=p+2$. We thus obtain
		\[W_{K}(G,H_0)\backslash W(G,H_0)\cong 
		K\backslash\varphi^{-1}((S^1_0)_S\coprod (S^2_0)_{S'})
		\cong (S^1_0)_S\coprod (S^2_0)_{S'},\]
		where
		\[S^1_0=\{(i~j)\in\fS_{p+1}:~1\leq i<j\leq p+1,~i+j=p+2\},\]
		\[S^2_0=\{\{(i~j),~(p+2-j~p+2-i)\}\subset\fS_{p+1}:~1\leq i<j\leq p+1,~i+j<p+2\}.\]
		For people who still hope to determine $I_0$ and $\Gamma\backslash I'_0$, we note that the inverse map of Proposition \ref{prop:S_0} (2) is given by $(i~j)\mapsto W_{K,0}(j~p+1)(i~p)$. This implies
		\[I_0=\{W_{K,0}(j~p+1)(i~p)\in W_{K,0}\backslash \fS_n:~1\leq i<j\leq p+1,~i+j=p+2\},\]
		\[\Gamma \backslash I'_0=\left\{\left\{\begin{array}{c}
			W_{K,0}(j~p+1)(i~p),\\
			W_{K,0}(p+2-i~p+1)(p+2-j~p)
		\end{array}\right\}\subset W_{K,0}\backslash \fS_n:~
		\begin{array}{l}
			1\leq i<j\leq p+1,\\
			i+j<p+2
		\end{array}
		\right\}.\]

		Set $S_1=\{e\}$. Then we have
		$\cI_{G,\theta}\otimes_k k'\cong (S_0)_{S'}\coprod (S_1)_{S'}$.
		It is evident that the Galois action on $(S_1)_{S'}\subset (\fS_n)_{S'}$ is trivial. We thus get
		$\cI_{G,\theta}\cong (S^1_0)_S\coprod (S^2_0)_{S'}\coprod (S_1)_{S}$.
		In view of the first paragraph, it remains to study
		\[K\backslash rt^{-1}((S_1)_S)\cong W_{K}(G,H_1)\backslash W(G,H_1).\]
		Note that $rt^{-1}((S_1)_S)$ is the moduli scheme of $\theta$-stable Borel subgroups.
		We plan to achieve this by computing $I_1$ and $\Gamma\backslash I'_1$ directly.
		
		Recall that $W_{K,1}=\fS_p$. The transpositions
		$(j~p+1)\in\fS_{p+1}$ ($1\leq j\leq p+1$) form a complete system of representatives of $\fS_p\backslash\fS_{p+1}$. For $1\leq j\leq p+1$, one can see by transposing $p+1$ that
		\[(j~p+1)w_0 (j~p+1)\in\begin{cases}
			\fS_p&(p~\mathrm{is}~\mathrm{even},~j=\frac{p}{2}+1)\\
			\fS_p(p+2-i~p+1)&(\mathrm{otherwise}).
		\end{cases}\]
		Hence we obtain
		\[I_1=\begin{cases}
			\emptyset&(p~\mathrm{is}~\mathrm{odd})\\
			\left\{\left(\frac{n}{2}+1~n+1\right)\right\}
			&(p~\mathrm{is}~\mathrm{even}),
		\end{cases}\]
		\[\Gamma\backslash I'_1=\begin{cases}
			\left\{\{(j~n+1),~(n+2-j~n+1)\}\subset\fS_n:~
			1\leq j\leq \frac{n+1}{2}\right\}&(p~\mathrm{is}~\mathrm{odd})\\
			\left\{\{(j~n+1),~(n+2-j~n+1)\}\subset\fS_n:~1\leq j\leq \frac{n}{2}\right\}
			&(p~\mathrm{is}~\mathrm{even}).
		\end{cases}\] 
	\end{ex}

	\section*{Declaration of competing interest}
	
	The author declares no competing interests.
	
	\section*{Data availability}
	
	No data was used for the research described in the article.

\end{document}